\documentclass[11pt,reqno]{amsart}

\usepackage{amsfonts,amssymb,amsmath,float,fullpage,graphicx}
\usepackage[usenames,dvipsnames]{color}
\usepackage{caption}
\usepackage{siunitx}
\usepackage{empheq}
\usepackage[mathscr]{euscript}
\usepackage[T1]{fontenc}
\usepackage{mathtools}
\usepackage{comment}
\usepackage{bm}
\usepackage{caption}
\usepackage{hyperref}
\usepackage{enumerate}
\usepackage[shortlabels]{enumitem}
\usepackage{tikz}
\usepackage{mathdots}
\usetikzlibrary{graphs,patterns,decorations.markings,arrows,matrix}
\usetikzlibrary{calc,decorations.pathmorphing,decorations.pathreplacing,shapes}
\renewcommand{\tikz}[2]{
\begin{tikzpicture}[scale=#1,baseline=(current bounding box.center),>=stealth]
#2
\end{tikzpicture}}
\allowdisplaybreaks

\newtheorem{thm}{Theorem}[section]
\newtheorem{lem}[thm]{Lemma}
\newtheorem{prop}[thm]{Proposition}

\theoremstyle{definition}

\theoremstyle{remark}
\newtheorem{rmk}[thm]{Remark}
\theoremstyle{remark}
\newtheorem{ex}[thm]{Example}
\numberwithin{equation}{section}

\def\leq{\leqslant}
\def\geq{\geqslant}
\def\i{{\tt i}}

\colorlet{llgray}{white!95!black}
\colorlet{lgray}{white!85!black}
\colorlet{lred}{white!85!red}
\colorlet{dred}{black!30!red}
\colorlet{lgreen}{white!80!green}
\colorlet{dgreen}{black!30!green}
\colorlet{dblue}{black!30!blue}
\definecolor{green}{rgb}{0.1,0.8,0.1}
\definecolor{yellow}{rgb}{1.0,0.85,0.25}

\DeclareMathOperator{\res}{res}

\newcommand{\N}[0]{\mathbb N}

\author{Jimmy He}
\address{Department of Mathematics, Ohio State University, Columbus, OH 43210}
\email{he.2673@osu.edu}

\author{Michael Wheeler}
\address{School of Mathematics and Statistics, University of Melbourne, Victoria 3010, Australia}
\email{wheelerm@unimelb.edu.au}

\date{\today}

\begin{document}

\title{Free boundary $q$-Whittaker and Hall--Littlewood processes}

\begin{abstract}
We study the free boundary $q$-Whittaker and Hall--Littlewood processes, two probability measures on sequences of partitions. We prove that a certain observable of the free boundary $q$-Whittaker process exhibits a $(q,t)$ symmetry after a random shift, generalizing a previous result of Imamura, Mucciconi, and Sasamoto, and an extension of that result due to the first author. Our proof is completely different, and as part of our proof, we find contour integral formulas for the free boundary $q$-Whittaker process. We also show a matching between certain observables in the free boundary Hall--Littlewood process and a quasi-open six vertex model, and explain how work of Finn and Vanicat gives an evaluation of a bounded sum over skew Hall--Littlewood functions as a rectangular Koornwinder polynomial.
\end{abstract}

\maketitle

\section{Introduction}
\subsection{Background}
Many recent advances in the study of integrable probabilistic systems have been driven by the study of probability measures on partitions defined in terms of symmetric functions. The first such results focused on Schur measures, which are related to zero-temperature growth models, see e.g. \cite{J00, BR01, SI04, BBCS18}. Schur measures have a fermionic (i.e. determinantal/Pfaffian) structure, which aids in asymptotics that are needed for probabilistic applications. More recently, connections were made between generalizations of Schur measures using Macdonald polynomials, and positive temperature models which have no obvious fermionic structure. These were then studied using hard but ad hoc methods, see e.g. \cite{BBC20, BBCW18, BC14, BCF14}. These ideas have also been used to study stationary measures for these systems, see e.g. \cite{CdGW15, CMW22, BCY24}.

Recently, Imamura, Mucciconi, and Sasamoto \cite{IMS21,IMS21b} discovered two new identities of symmetric functions relating full- and half-space $q$-Whittaker measures to periodic and free boundary Schur measures, leading to new asymptotic results for half-space models in the Kardar-Parisi-Zhang universality class \cite{DLM25,H23, IMS22}. The periodic Schur measure, introduced by Borodin \cite{B07}, and free boundary Schure measure, introduced by Betea, Bouttier, Nejjar, and Vuleti\'c \cite{BBNV18}, are both fermionic. Thus, these identities revealed that $q$-Whittaker measures, while not fermionic, are intimately related with fermionic measures. Perhaps even more surprising than the identities themselves is that the proof was via an intricate bijection. Let us also mention that there have many further developments in the study of half-space models, see e.g. \cite{DS25,BCD25,DY25,DZ24,G24}.

In a previous work \cite{HW23}, we found an an extension of one of the identities of \cite{IMS21}, which related the full space $q$-Whittaker measure to the periodic Schur measure, recasting the relationship as a symmetry of the periodic $q$-Whittaker measure, a natural generalization of both measures. The periodic $q$-Whittaker measure has two parameters, and reduces to either the full space $q$-Whittaker measure or the periodic Schur measure depending on which is specialized to $0$. These two parameters have a non-trivial symmetry which exchanges them without changing the measure, after a simple shift. This naturally begs the question of whether there is a free boundary $q$-Whittaker measure which shares a similar symmetry property. 

In this work, we give a definition of the free boundary $q$-Whittaker measure which naturally extends the existing half space $q$-Whittaker measure and free boundary Schur measure. We show that this measure shares the same symmetry property as the periodic $q$-Whittaker measure, generalizing the other identity of \cite{IMS21}. In particular, our proof is new even in this special case, and is much shorter than the original proof of \cite{IMS21} which used an intricate bijection. Of course there are many benefits of having a bijective proof, and a natural question is whether the bijection of \cite{IMS21} could be extended to prove our more general identities.

Our proof (as in \cite{HW23}) involves finding explicit contour integral formulas for the distribution of the largest part in the free boundary $q$-Whittaker measure, after a suitable shift. These formulas are themselves striking, exhibiting many symmetries in the integration variables. Moreover, establishing these formulas requires new ideas compared to \cite{HW23}, as is usually the case when considering models with open boundaries. In particular, while in \cite{HW23} we were able to directly establish a contour integral formula exhibiting the symmetry of parameters, in the free boundary setting the ideas of \cite{HW23} lead to a formula which has none of the desired symmetry properties. We must conduct an intricate analysis of the residues, and partially evaluate the integrals in order to obtain the final formula which does exhibit the symmetries we wish to establish.

A second result of this paper is to establish a relationship between free boundary measures and what we call quasi-open models (beyond the Schur case \cite{BBNV18}). We define a version of the stochastic six vertex on a slightly unusual geometry, which we call the quasi-open six vertex model, and give a relationship to the free boundary Hall--Littlewood process. Similar results have been obtained in the full space \cite{BBW16}, half-space \cite{BBCW18, H23}, and periodic \cite{HW23} settings, and connections between $q$-Whittaker measures and integrable probabilistic models are also known, see e.g. \cite{OSZ14, MP17, BBC20}. So far, probabilistic results on periodic and free boundary measures, and their related stochastic models, have been restricted to the Schur setting. Nevertheless, we believe this connection to integrable stochastic models is important, and hope that our results may be useful in initiating the study of these systems beyond the Schur case.

Finally, let us remark that this work fits naturally into a line of work on bounded or refined Littlewood identities. Littlewood identities evaluate sums of a single symmetric function (as opposed to a product in the Cauchy identity), usually in terms of a factored expression. There have been many cases where these identities have either been refined with additional parameters \cite{WZJ16, G23, BW16, BWZJ15}, or bounded by restricting the sum \cite{RW21, HKKO25}. The tail probabilities of the measures we study involve expressions exactly like this except that we sum over skew functions, and in this paper we also point out how a result of Finn and Vanicat \cite{FV17} gives an evaluation for a bounded sum over skew Hall--Littlewood functions in terms of a rectangular Koornwinder polynomial. 

\subsection{Main results}

\subsubsection{Rogers--Szeg\H{o} polynomials}

Fix an indeterminate $z$ and nonnegative integer $m \in \mathbb{Z}_{\geq 0}$.
The Rogers--Szeg\H{o} polynomials are defined as
\begin{align}
\label{RS}
h_m(z;q)
=
\sum_{k=0}^{m} {m \choose k}_q z^k,
\end{align}
where we have used the standard Gaussian $q$-binomial
\begin{align*}
{m \choose k}_q
=
\frac{(q;q)_m}{(q;q)_k (q;q)_{m-k}},
\qquad
0 \leq k \leq m.
\end{align*}

\subsubsection{The function $h_{\lambda}$}

Fix a partition $\lambda = 1^{m_1} 2^{m_2} 3^{m_3} \dots$ written in multiplicative form, and two indeterminates $a,b$. We define (see \cite{W06})
\begin{align}
\label{RS-product}
h_\lambda(a,b;q)
=
\prod_{i \geq 1}
a^{m_{2i-1}} h_{m_{2i-1}}(b/a;q)
h_{m_{2i}}(ab;q),
\end{align}
with $h_m(z;q)$ denoting the Rogers--Szeg\H{o} polynomial \eqref{RS}, and where the above product terminates due to the fact that $\lambda$ has finitely many nontrivial part multiplicities.

Note that because $m_k = \lambda_k'-\lambda_{k+1}'$ for all $k \geq 1$, we can also write the \eqref{RS-product} as
\begin{align*}
h_\lambda(a,b;q)
=
\prod_{i \geq 1}
a^{\lambda'_{2i-1}-\lambda'_{2i}} h_{\lambda'_{2i-1}-\lambda'_{2i}}(b/a;q)
h_{\lambda'_{2i}-\lambda'_{2i+1}}(ab;q),
\end{align*}
which leads to an elegant analogue of \eqref{RS-product} for conjugated partitions:
\begin{align*}
h_{\lambda'}(a,b;q)
=
\prod_{i \geq 1}
a^{\lambda_{2i-1}-\lambda_{2i}} h_{\lambda_{2i-1}-\lambda_{2i}}(b/a;q)
h_{\lambda_{2i}-\lambda_{2i+1}}(ab;q).
\end{align*}

\subsubsection{Results on $q$-Whittaker processes}

In what follows, let $P_{\lambda/\mu}(x;q,0)$ denote a skew $q$-Whittaker polynomial in an alphabet $x = (x_1,\dots,x_N)$ of arbitrary (possibly infinite) size; see \cite[Chapter VI, Section 7]{M79} for the definition of skew Macdonald polynomials $P_{\lambda/\mu}(x;q,t)$ and \cite{BW21} for explicit information regarding the $t=0$ specialization.

\begin{thm}
\label{thm: qt symmetry}
Fix an integer $n \geq 0$. The quantity
\begin{align}
\label{conjecture}
Z_{n}(x;q,t;a,b,c,d)
=
\sum_{\mu \subseteq \lambda:\lambda_1 \leq n}
\frac{h_{n-\lambda_1}(ab;q)h_{\lambda'}(a,b;q)}{(q;q)_{n-\lambda_1}\prod_{i \geq 1}(q;q)_{\lambda_i-\lambda_{i+1}}}
\cdot
P_{\lambda/\mu}(x;q,0)
\cdot
t^{|\mu|/2}
h_{\mu'}(c/\sqrt{t},d/\sqrt{t};q)
\end{align}
is separately symmetric in the parameters $(q,t)$ and $(a,b,c,d)$.
\end{thm}

\begin{rmk}
Theorem \ref{thm: qt symmetry} generalize Theorem 10.12 in \cite{IMS21}, as well as its extension Theorem 3.1 of \cite{H23}. These special cases are given by $Z_{n}(x;q,0,a,0,\sqrt{q}a,0)=Z_{n}(x;0,q,a,\sqrt{q}a,0,0)$ and $Z_{n}(x;q,0,a,0,\sqrt{q}b,0)=Z_{n}(x;0,q,a,\sqrt{q}b,0,0)$ respectively. Our proof is new even in these special cases, and is much shorter than the original proof in \cite{IMS21}.
\end{rmk}

\begin{rmk}
After suitably normalizing, this theorem has an interpretation in terms of random partitions. The right hand side can be interpreted as the probability that $\lambda_1+\chi\leq n$, where $\lambda$ is a random partition distributed according to the free boundary $q$-Whittaker measure, see Section \ref{sec: fb qw and hl} for a definition, and $\chi$ is an independent random variable distributed whose distribution function is given by $\mathbb{P}(\chi\leq n)=(q;q)_\infty(ab;q)_\infty\frac{h_n(ab;q)}{(q;q)_n}$. In fact, $\chi$ can be viewed as the sum of two $q$-geometric random variables with parameters $q$ and $ab$. This random shift appeared previously in \cite{H23}.
\end{rmk}

\begin{ex}
For $n=1$ and $N=1$, the quantity \eqref{conjecture} is small enough to compute explicitly. In that case, the sum is over all pairs of partitions $\lambda,\mu$ such that $0 \leq \mu_1 \leq \lambda_1 \leq 1$, and $0 \leq \ell(\lambda)-\ell(\mu) \leq 1$. The possible $(\lambda,\mu)$ pairs that satisfy these constraints are (a) $\lambda = \mu =\varnothing$; (b1) $\lambda = \mu = (1)$; (b2) $\lambda = (1), \mu = \varnothing$; (c1) $\lambda=\mu=(1^k)$ with $k \geq 2$; (c2) $\lambda=(1^k), \mu= (1^{k-1})$ with $k \geq 2$. Documenting the summand for all such cases, we have
\begin{align*}
&
{\rm (a)} : \dfrac{1+ab}{1-q}
\qquad\qquad
{\rm (b1)} + {\rm (b2)} : \dfrac{a+b}{1-q}(c+d+x_1)
\\
&
{\rm (c1)} + {\rm (c2)} : 
\Bigg[
\dfrac{1+ab}{1-q}(1+cd/t+x_1(c+d)/t)
+
\dfrac{a+b}{1-q}(c+d+x_1(cd/t+1))
\Bigg]
\sum_{i=1}^{\infty}t^i
\end{align*}
where the result for ${\rm (c1)} + {\rm (c2)}$ follows by keeping track, separately, of even and odd weights for $\lambda=(1^k)$. Adding all of these terms, one finds that
\begin{multline*}
Z_{1}(x_1;q,t;a,b,c,d)
=
\\
\frac{1+ab+ac+ad+bc+bd+cd+abcd+(a+b+c+d+abc+abd+acd+bcd)x_1}{(1-q)(1-t)},
\end{multline*}
which indeed exhibits the required symmetry in $(q,t)$ and $(a,b,c,d)$.
\end{ex}

\subsubsection{Results on Hall--Littlewood processes}
Let $P_{\lambda/\mu}(x;0,t)$ denote a skew Hall--Littlewood polynomial in the alphabet $x=(x_1,\dotsc, x_N)$ and define the \emph{free boundary Hall--Littlewood process} as the (possibly signed) probability measure on sequences of partitions $\vec\lambda=(\lambda^{(0)}\subseteq\dotsc\subseteq \lambda^{(N)})$ proportional to
\begin{equation*}
    \frac{q^{|\lambda^{(0)}|/2}h_{(\lambda^{(0)})'}(a,b;q)h_{(\lambda^{(N)})'}(c/\sqrt{u},d/\sqrt{u};q)}{\prod_j (q;q)_{\lambda^{(N)}_j-\lambda^{(N)}_{j+1}}}\prod_{i=1}^N P_{\lambda^{(i)}/\lambda^{(i-1)}}(x_i;0,t)
\end{equation*}

We now informally define the quasi-open six vertex model, see Section \ref{sec: quasi-open 6vm} for a formal definition. The usual six vertex model on a triangle is a probability distribution on arrows traveling the edges of half an $N\times N$ square lattice. The vertices on the diagonal boundary (of degree 2) play a special role, since arrows may enter/exit the model at these locations. Each edge supports at most one arrow, and the probability is a product of local weights depending on many parameters, including rapidities $x_i$ associated to rows/columns, an anisotropy parameter $t$, and boundary parameters $a,b$. The local weight at $(i,j)$ depends on $x_i$ and $x_j$, and either $t$ or $a,b$ depending on whether $i\neq j$ or $i=j$.

The quasi-open six vertex model is a version of the six vertex model on a strip, built out of an infinite collection of these triangles in an alternating fashion. The model now has two boundaries (the two sides of the strip), and so has four boundary parameters $a,b,c,d$. In addition, it has an additional parameter $q$, such that the $i$th triangle has rapidities which are modified by $q^{i/2}$. Since $q<1$, the geometric decay ensures that eventually, all arrows enter from the left boundary, travel horizontally, and exit the right boundary. The arrows enter from infinity, and eventually exit the last triangle. See Figure \ref{fig: quasi-open 6vm} for an example of (part of) a configuration of this model.

\begin{figure}
    \centering
			\begin{tikzpicture}[scale=0.6,baseline=(current bounding box.center),>=stealth]
			\draw[lgray, fill](-0.5,-0.5)--(2.5,2.5)--(-0.5,2.5)--cycle;
			\draw[lgray, fill](-0.5,3)--(2.5,6)--(2.5,3)--cycle;
			\draw[lgray, fill](3,3)--(6,6)--(3,6)--cycle;
			
			\draw[thick, dotted](-0.5,0)--(0,0)--(0,3.5)--(3.5,3.5)--(3.5,6);
			\draw[thick, dotted](-0.5,1)node[left=0.5cm]{$\dots$}--(1,1)--(1,4.5)--(4.5,4.5)--(4.5,6);
			\draw[thick, dotted](-0.5,2)--(2,2)--(2,5.5)--(5.5,5.5)--(5.5,6);
			
			\draw[ultra thick, ->](-0.5,0)--(0.25,0);
			
			\draw[ultra thick, ->](-0.5,1)--(1,1)--(1,3.5)--(2,3.5)--(2,5.5)--(3.5,5.5)--(3.5,6);
			
			\draw[ultra thick, ->](-0.5,2)--(0,2)--(0,3.75);
			
			\draw[ultra thick, ->](4.5,4.5)--(4.5,5.5)--(5.75,5.5);
			
			\end{tikzpicture}
    \caption{Example of part of a configuration of the quasi-open six vertex model. Note that arrows enter/exit from the diagonal boundaries.}
    \label{fig: quasi-open 6vm}
\end{figure}
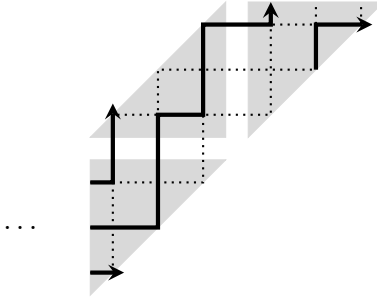

Out second main result relates the free boundary Hall--Littlewood process to the quasi-open six vertex model. We study an observable measuring the deviation of the system from the state where all horizontal edges are occupied, and all vertical ones unoccupied, and show that this, along with the locations that arrows exit the final triangle, are matched to the free boundary Hall--Littlewood process (Theorem \ref{thm: HL 6vm}). Our proof uses another integrable model of deformed bosons, and Yang--Baxter graphical arguments. Finally, we also state a result which is essentially found in \cite{FV17}, which equates a bounded sum over skew Hall--Littlewood polynomials similar to $Z_n(x;q,t;a,b,c,d)$ to a rectangular Koornwinder polynomial.

\subsection{Further directions}
Our proof, while direct, does not seem to shed light on the nature of these identities. It would be interesting to see if the bijective approach of \cite{IMS21} could be extended to prove the identities of this paper and in \cite{HW23}. It would also be interesting to understand if there is some algebraic structure behind the symmetries we have found.

A natural next step would be to try and study asymptotics for the quasi-open (and also quasi-periodic) six vertex model, or other related models. Unfortunately, our formulas as they stand do not seem amenable to asymptotic analysis. We believe it would be an interesting direction to obtain asymptotic results, whether through our formulas or other means, but we leave this as an open question.

Finally, it is unclear at this moment exactly how ubiquitous these types of symmetries and identities are. It would be interesting to find additional examples with other families of symmetric functions, and this may also shed light on exactly where these identities come from.

\subsection{Organization}
The paper is organized as follows. In Section \ref{sec: prelim}, we review some background on Macdonald polynomials and formally define the free boundary $q$-Whittaker and Hall--Littlewood processes. In Section \ref{sec: extra sym}, we establish some non-trivial symmetries of $Z_n$, one of which reduces Theorem \ref{thm: qt symmetry} to the case $a=b=c=d=0$. In Section \ref{sec: main pf}, we state and prove contour integral formulas for $Z_n$ which imply Theorem \ref{thm: qt symmetry}. Finally, in Section \ref{sec: HL process}, we formally define the quasi-open six vertex model, and show a distributional matching to the free boundary Hall--Littlewood process, and give a related identity showing a bounded sum over skew Hall--Littlewood polynomials equals a rectangular Koornwinder polynomial.

\section{Preliminaries}

\label{sec: prelim}
\subsection{Macdonald polynomials}
We refer the reader to \cite{M79} for further background on symmetric functions and Macdonald polynomials.

A \emph{partition} is a finite non-increasing sequence of non-negative integers $\lambda=(\lambda_1\geq \lambda_2\geq \dotsm\geq \lambda_k)$, and we call $k$ the \emph{length}, denoted $l(\lambda)$. We let $|\lambda|=\sum_{i\geq 1} \lambda_i$ denote its \emph{size}. It is frequently useful to view a partition as its \emph{Young diagram}, where we place $\lambda_i$ boxes in the $i$th row from top to bottom, so that each row is aligned on the left. We define the \emph{conjugate partition} $\lambda'$ as the partition obtained from $\lambda$ by reflecting its Young diagram, switching the rows and columns. Given a box $s\in\lambda$, we let $a(s)$ and $l(s)$ be the \emph{arm length} and \emph{leg length}, defined as the number of boxes to the right and below respectively. We will let $m_i(\lambda)$ denote the number of occurrences of $i$ in $\lambda$.

Let $x=(x_1,x_2,\dotsc)$ be a formal alphabet, and let $\Lambda$ denote the ring of symmetric functions. We let $P_\lambda(x;q,t)$ denote the \emph{Macdonald polynomials}, defined as the unique symmetric functions orthogonal with respect to the inner product defined by
\begin{equation*}
    \langle p_\lambda,p_\mu\rangle=\delta_{\lambda,\mu}z_\lambda\prod_i\frac{1-q^{\lambda_i}}{1-t^{\lambda_i}},
\end{equation*}
where $p_\lambda=\prod_{i\geq 1} p_{\lambda_i}$, $p_k=\sum_{i\geq 1}x_i^k$, and $z_\lambda=\prod _i m_i(\lambda)!i^{m_i(\lambda)}$, and whose change of basis to the monomial symmetric functions is upper triangular with respect to the dominance ordering on partitions (see Chapter VI, Section 4 of \cite{M79}). We let $Q_{\lambda}(x;q,t)$ denote the dual basis with respect to this inner product, and define $b_\lambda(q,t)$ by $Q_\lambda(x;q,t)=b_\lambda(q,t) P_\lambda(x;q,t)$. We have
\begin{equation*}
b_\lambda(q,t)=\prod_{s\in\lambda}\frac{1-q^{a(s)}t^{l(s)+1}}{1-q^{a(s)+1}t^{l(s)}},
\end{equation*}
and note that $b_\lambda(q,t)=b_{\lambda'}(t,q)$. Macdonald polynomials satisfy a \emph{Cauchy identity}, which states that
\begin{equation*}
    \sum_{\lambda}P_\lambda(x;q,t)Q_{\lambda}(y;q,t)=\prod_{i,j}\frac{(tx_iy_j;q)_\infty}{(x_iy_j;q)_\infty}=:\Pi(x,y;q,t).
\end{equation*}

For a skew Young diagram $\lambda/\mu$, the \emph{skew Macdonald polynomials} $P_{\lambda/\mu}(x;q,t)$ are then defined by
\begin{equation*}
    \langle P_{\lambda/\mu},Q_\nu\rangle=\langle P_\lambda, Q_\mu Q_\nu\rangle
\end{equation*}
for all $Q_\nu$. The $Q_{\lambda/\mu}$ are similarly defined, with the roles of $P$ and $Q$ swapped. These are homogeneous polynomials in the alphabet $x$ of degree $|\lambda|-|\mu|$. They satisfy a \emph{branching rule}, meaning that if we specialize into two sets of variables $(a,b)$, then
\begin{equation*}
    P_{\lambda/\mu}(a,b;q,t)=\sum_{\nu}P_{\lambda/\nu}(a;q,t)P_{\nu/\mu}(b;q,t),
\end{equation*}
and similarly for the dual family $Q_{\lambda/\mu}$.

Recall that $\Lambda$ is isomorphic to a polynomial ring in variables $p_k$, so any homomorphism from $\Lambda$ is uniquely defined by a choice of where to send $p_k$ for each $k$. There is an important involution $\omega_{q,t}$ on the ring of symmetric functions defined by
\begin{equation*}
    \omega_{q,t}(p_\lambda)=(-1)^{|\lambda|-l(\lambda)}\prod_i\frac{1-q^{\lambda_i}}{1-t^{\lambda_i}}p_\lambda.
\end{equation*}
This involution satisfies $\omega_{q,t}(P_{\lambda/\mu}(x;q,t))=Q_{\lambda'/\mu'}(x;t,q)$, and $\omega_{q,t}(Q_{\lambda/\mu}(x;q,t))=P_{\lambda'/\mu'}(x;t,q)$, simultaneously swapping the roles of $P$ and $Q$, along with $q$ and $t$, while also conjugating all partitions.

We will be interested in two special cases of the Macdonald polynomials. When $q=0$, the Macdonald polynomials are called the \emph{Hall--Littlewood polynomials}, and when $t=0$, they are called the \emph{$q$-Whittaker polynomials}. The common specialization $q=t=0$ (actually $q=t$ suffices) are called the \emph{Schur polynomials}, which we will denote $s_{\lambda}(x)$.

Finally, a key fact which will be used in this paper is the following Littlewood identity due to Warnaar \cite[Theorem 1.1]{W06}.
\begin{thm}
\label{thm: littlewood}
Fix an integer $n \geq 1$ and let $y=(y_1,\dots,y_n)$ be an alphabet. The Hall--Littlewood polynomials obey the Littlewood identity
\begin{align}
\label{littlewood-rs}
\sum_{\nu}
h_{\nu}(a,b;q)
P_{\nu}(y;0,q)
=
\prod_{i=1}^{n}
\frac{(1+ay_i)(1+by_i)}{1-y_i^2}
\prod_{1 \leq i < j \leq n}
\frac{1-ty_i y_j}{1-y_i y_j},
\end{align}
which may be interpreted either formally or as a convergent series if one assumes that $|y_i y_j|<1$ for all $1 \leq i<j \leq n$.
\end{thm}
The special case when $a=b=c=d=0$ is given by
\begin{align}
\label{littlewood}
\sum_{\nu \text{ even}}
P_{\nu}(y;0,t)
=
\prod_{i=1}^{n}
\frac{1}{1-y_i^2}
\prod_{1 \leq i < j \leq n}
\frac{1-ty_i y_j}{1-y_i y_j}.
\end{align}

\subsection{Free boundary $q$-Whittaker and Hall--Littlewood processes}
\label{sec: fb qw and hl}
We now introduce and define the free boundary $q$-Whittaker and Hall--Littlewood processes. These are a natural common generalization of the half space measures $q$-Whittaker and Hall--Littlewood processes appearing in \cite{BBC20, BBCW18, IMS22, H23}, and the free boundary Schur process introduced in \cite{BBNV18}. Note that by duality, it suffices to just define one. One could attempt to define a Macdonald analogue, but the boundary weights given below would need to be changed, see Remark \ref{rmk: fb mac}.

Fix parameters $q,t\in [0,1)$ and $x=(x_1,\dotsc, x_N)$ with $x_i\in [0,1)$. We also fix parameters $a,b,c,d$, which we may allow to be negative, see Remark \ref{rmk: signed meas}. We define the \emph{free boundary Hall--Littlewood process} as a probability measure (possibly signed) on a sequence of partitions $\vec\lambda=(\lambda^{(0)}\subseteq \lambda^{(1)}\subseteq \dotsc \subseteq \lambda^{(N)})$ by
\begin{equation*}
    \mathbb{FBHL}_{q,t,a,b,c,d}^x(\vec\lambda)=\frac{1}{\Phi_{HL}}\frac{q^{|\lambda^{(0)}|/2}h_{\lambda^{(0)}}(a,b;t)h_{\lambda^{(N)}}(c/\sqrt{q},d/\sqrt{q};t)}{\prod_j (t;t)_{m_j(\lambda^{(0)})}}\prod_{i=1}^N P_{\lambda^{(i)}/\lambda^{(i-1)}}(x_i;0,t).
\end{equation*}
By applying the Macdonald involution, we can define the \emph{free boundary $q$-Whittaker process} by
\begin{equation*}
    \mathbb{FBW}_{u,q,a,b,c,d}^x(\vec\lambda)=\frac{1}{\Phi_{W}}\frac{u^{|\lambda^{(0)}|/2}h_{(\lambda^{(0)})'}(a,b;q)h_{(\lambda^{(N)})'}(c/\sqrt{u},d/\sqrt{u};q)}{\prod_j (q;q)_{\lambda^{(N)}_j-\lambda^{(N)}_{j+1}}}\prod_{i=1}^N P_{\lambda^{(i)}/\lambda^{(i-1)}}(x_i;q,0).
\end{equation*}
Here $\Phi_{HL}$ and $\Phi_{W}$ are constants (depending on all parameters) such that these measures have total mass $1$, see Theorem \ref{thm:pf}. We will call the $\lambda^{(N)}$ marginal of these measures the \emph{free boundary $q$-Whittaker measure} and \emph{free boundary Hall--Littlewood measure} respectively, and write the corresponding measures with an $M$ at the end (so for example, $\mathbb{FBHLM}_{q,t,a,b,c,d}^x$ denotes the free boundary Hall--Littlewood measure).

\begin{rmk}
\label{rmk: fb mac}
Since half space measures have been defined with Macdonald polynomials, see \cite{BBC20}, it is natural to ask whether these free boundary measures make sense with Macdonald polynomials. In fact, naively putting Macdonald polynomials into the above expressions will not lead to anything nice, since the $h_\lambda$ are only related to the two specializations of Macdonald polynomials, see \cite{W06} for a more detailed discussion. But by setting $a=b=c=d=0$, there should be no issue in defining a Macdonald analogue and for example, the partition function could be computed using a skew Littlewood identity for Macdonald polynomials, see e.g. Equation (1.6) in \cite{R12}. Perhaps more interesting, one could possibly keep two parameters, as Littlewood identities with a single boundary parameter are known. Since we do not study these more general measures, we will not formally define them here.
\end{rmk}

\begin{rmk}
\label{rmk: signed meas}
Given that the parameters $a,b,c,d$ may ultimately be absorbed into the variables $x$, see Proposition \ref{prop: abcd equiv}, it is natural to ask why we even include them. These parameters turn out to be very natural, especially in the Hall--Littlewood case, due to connections to integrable probabilistic models like the six vertex model (see Theorem \ref{thm: HL 6vm}), and also because of a relation to Koornwinder polynomials, see Theorem \ref{thm: koornwinder}, originally proved in \cite{FV17}. Since these parameters seem to play a special role, and their equivalence with the variables $x$ is not obvious, we have decided to retain them.

However, the connections with probabilistic systems seems to require taking two of the parameters in $a,b,c,d$ to be negative. In some cases, this means that we no longer have a true probability measure, but a signed probability measure. We will continue to use probabilistic notation, but these statements should be interpreted appropriately. For example, we will continue to write the sum of random variables, which must be interpreted as convolution of signed measures.
\end{rmk}

We end this section by computing the partition function $\Phi_W$ for the free boundary $q$-Whittaker process. The Hall--Littlewood partition function $\Phi_{HL}$ may be computed by applying the Macdonald involution. Theorem \ref{thm: HL 6vm} also gives a separate argument for this, see Remark \ref{rmk: HL partition fn}.

\begin{thm}
\label{thm:pf}
Fix an alphabet $x=(x_1,\dots,x_N)$. The following summation identity holds:
\begin{multline}
\label{partition-function}
\sum_{\mu \subseteq \lambda}
\frac{h_{\lambda'}(a,b;q)}
{\prod_{i \geq 1}(q;q)_{\lambda_i-\lambda_{i+1}}}
\cdot
P_{\lambda/\mu}(x;q,0)
\cdot
t^{|\mu|/2}
h_{\mu'}(c/\sqrt{t},d/\sqrt{t};q)
=
\\
\prod_{i=1}^{N}
\frac{1}{(ax_i;q,t)_{\infty}(bx_i;q,t)_{\infty}(cx_i;q,t)_{\infty}(dx_i;q,t)_{\infty}}
\prod_{1 \leq i<j \leq N}
\frac{1}{(x_i x_j;q,t)_{\infty}}
\\
\times
\frac{(ab;q)_{\infty}}{(t;t)_{\infty}(qt;q,t)_{\infty}(ab;q,t)_{\infty}(ac;q,t)_{\infty}(ad;q,t)_{\infty}(bc;q,t)_{\infty}(bd;q,t)_{\infty}(cd;q,t)_{\infty}}.
\end{multline}
\end{thm}

\begin{proof}
Our starting point is the Littlewood identity \eqref{littlewood-rs}, which we extend to an alphabet $y$ of unspecified size:
\begin{align*}
\sum_{\lambda}
h_{\lambda}(a,b;q)
P_{\lambda}(y;0,q)
=
\prod_{i \geq 1}
\frac{(1+ay_i)(1+by_i)}{1-y_i^2}
\prod_{i < j}
\frac{1-qy_i y_j}{1-y_i y_j}.
\end{align*}
Writing $y = (z,w)$ (that is, replacing $y$ by the union of two smaller alphabets $z$, $w$) and applying the branching rule on the left hand side, we have
\begin{multline}
\label{skew-littlewood-deriv}
\sum_{\lambda}
\sum_{\mu}
h_{\lambda}(a,b;q)
P_{\lambda/\mu}(z;0,q)
P_{\mu}(w;0,q)
=
\\
\prod_{i \geq 1}
\frac{(1+az_i)(1+bz_i)}{1-z_i^2}
\prod_{i < j}
\frac{1-qz_i z_j}{1-z_i z_j}
\cdot
\prod_{i,j \geq 1}
\frac{1-qz_i w_j}{1-z_i w_j}
\cdot
\prod_{i \geq 1}
\frac{(1+aw_i)(1+bw_i)}{1-w_i^2}
\prod_{i < j}
\frac{1-qw_i w_j}{1-w_i w_j}.
\end{multline}
The second last product (that depends on $z$ and $w$) in \eqref{skew-littlewood-deriv} may be expanded using the Cauchy identity, while the final product (that depends on $w$ only) may be expanded by again using the Littlewood identity. This results in the equation
\begin{multline}
\label{skew-littlewood-deriv2}
\sum_{\lambda}
\sum_{\mu}
h_{\lambda}(a,b;q)
P_{\lambda/\mu}(z;0,q)
P_{\mu}(w;0,q)
=
\\
\prod_{i \geq 1}
\frac{(1+az_i)(1+bz_i)}{1-z_i^2}
\prod_{i < j}
\frac{1-qz_i z_j}{1-z_i z_j}
\cdot
\sum_{\nu}
Q_{\nu}(z;0,q)
P_{\nu}(w;0,q)
\cdot
\sum_{\kappa}
h_{\kappa}(a,b;q)
P_{\kappa}(w;0,q).
\end{multline}
Using the multiplication rule for Hall--Littlewood polynomials \cite[Chapter 3, Section 3]{M79}, 
\begin{align*}
P_{\nu}(w;0,q) P_{\kappa}(w;0,q)
=
\sum_{\rho}
f^{\rho}_{\nu\kappa}(q)
P_{\rho}(w;0,q),
\end{align*}
one then has that
\begin{multline}
\label{skew-littlewood-deriv3}
\sum_{\lambda}
\sum_{\mu}
h_{\lambda}(a,b;q)
P_{\lambda/\mu}(z;0,q)
P_{\mu}(w;0,q)
=
\\
\prod_{i \geq 1}
\frac{(1+az_i)(1+bz_i)}{1-z_i^2}
\prod_{i < j}
\frac{1-qz_i z_j}{1-z_i z_j}
\cdot
\sum_{\nu}
\sum_{\kappa}
\sum_{\rho}
h_{\kappa}(a,b;q)
f^{\rho}_{\nu\kappa}(q)
Q_{\nu}(z;0,q)
P_{\rho}(w;0,q).
\end{multline}
Extracting coefficients of $P_{\mu}(w;0,q)$ on both sides of \eqref{skew-littlewood-deriv3} and using \cite[Chapter 3, Equation (5.2)]{M79}
\begin{align*}
\sum_{\nu}
f^{\mu}_{\nu\kappa}(q)
Q_{\nu}(z;0,q)
=
Q_{\mu/\kappa}(z;0,q),
\end{align*}
we arrive at the identity
\begin{multline}
\label{skew-littlewood-0}
\sum_{\lambda}
h_{\lambda}(a,b;q)
P_{\lambda/\mu}(z;0,q)
=
\prod_{i \geq 1}
\frac{(1+az_i)(1+bz_i)}{1-z_i^2}
\prod_{i < j}
\frac{1-qz_i z_j}{1-z_i z_j}
\cdot
\sum_{\kappa}
h_{\kappa}(a,b;q)
Q_{\mu/\kappa}(z;0,q),
\end{multline}
which may be viewed as a skew version of the original Littlewood identity \eqref{littlewood-rs}. If one multiplies both sides of \eqref{skew-littlewood-0} by $t^{|\mu|/2} h_{\mu}(c/\sqrt{t},d/\sqrt{t};q) / b_{\mu}(0,q)$ and sums over $\mu$, the resulting left hand side
\begin{align}
\label{iterate}
\Lambda(z;q,t;a,b,c,d)
:=
\sum_{\lambda}
\sum_{\mu}
\frac{h_{\lambda}(a,b;q)}{b_{\lambda}(0,q)}
\cdot
Q_{\lambda/\mu}(z;0,q)
\cdot
t^{|\mu|/2} h_{\mu}(c/\sqrt{t},d/\sqrt{t};q)
\end{align}
is then seen to obey the recurrence
\begin{align*}
\Lambda(z;q,t;a,b,c,d)
=
\prod_{i \geq 1}
\frac{(1+az_i)(1+bz_i)}{1-z_i^2}
\prod_{i < j}
\frac{1-qz_i z_j}{1-z_i z_j}
\cdot
\Lambda(\sqrt{t} z;q,t;c/\sqrt{t},d/\sqrt{t},a \sqrt{t}, b \sqrt{t}).
\end{align*}
The solution of this recurrence is
\begin{align*}
\Lambda(z;q,t;a,b,c,d)
=
\prod_{i \geq 1}
\frac{(-az_i;t)_{\infty}(-bz_i;t)_{\infty}(-cz_i;t)_{\infty}(-dz_i;t)_{\infty}}{(z_i^2;t)_{\infty}}
\prod_{i<j}
\frac{(qz_i z_j;t)_{\infty}}{(z_i z_j;t)_{\infty}}
\cdot
\Lambda(0;q,t;a,b,c,d),
\end{align*}
where
\begin{multline}
\label{normalization-compute}
\Lambda(0;q,t;a,b,c,d)
=
\sum_{\lambda}
t^{|\lambda|/2}
\frac{h_{\lambda}(a,b;q)h_{\lambda}(c/\sqrt{t},d/\sqrt{t};q)}
{b_{\lambda}(0,q)}
\\
=
\prod_{k=1}^{\infty}
\sum_{m=0}^{\infty}
t^{km}
\frac{h_m(ab;q) h_m(cd/t;q)}{(q;q)_m}
\cdot
\prod_{\ell=1}^{\infty}
\sum_{m=0}^{\infty}
(act^{\ell-1})^m
\frac{h_m(b/a;q) h_m(d/c;q)}{(q;q)_m}.
\end{multline}
The sums appearing on the right hand side of \eqref{normalization-compute} can be computed using Mehler's identity:
\begin{align}
\label{mehler}
\sum_{m=0}^{\infty}
\frac{h_m(\alpha;q) h_m(\beta;q)}{(q;q)_m}
u^m
=
\frac{(\alpha\beta u^2;q)_{\infty}}
{(u;q)_{\infty}(\alpha u;q)_{\infty}(\beta u;q)_{\infty} (\alpha \beta u;q)_{\infty}}.
\end{align}
Application of \eqref{mehler} in \eqref{normalization-compute} yields
\begin{align*}
&
\prod_{k=1}^{\infty}
\sum_{m=0}^{\infty}
t^{km}
\frac{h_m(ab;q) h_m(cd/t;q)}{(q;q)_m}
=
\prod_{k=1}^{\infty}
\frac{(abcdt^{2k-1};q)_{\infty}}
{(t^k;q)_{\infty}(abt^k;q)_{\infty}
(cdt^{k-1};q)_{\infty}(abcdt^{k-1};q)_{\infty}},
\\
\\
&
\prod_{\ell=1}^{\infty}
\sum_{m=0}^{\infty}
(act^{\ell-1})^m
\frac{h_m(b/a;q) h_m(d/c;q)}{(q;q)_m}
=
\prod_{\ell=1}^{\infty}
\frac{(abcdt^{2\ell-2};q)_{\infty}}
{(act^{\ell-1};q)_{\infty}(bct^{\ell-1};q)_{\infty}(adt^{\ell-1};q)_{\infty}(bdt^{\ell-1};q)_{\infty}},
\end{align*}
and we finally conclude that
\begin{multline*}
\Lambda(z;q,t;a,b,c,d)
=
\prod_{i \geq 1}
\frac{(-az_i;t)_{\infty}(-bz_i;t)_{\infty}(-cz_i;t)_{\infty}(-dz_i;t)_{\infty}}{(z_i^2;t)_{\infty}}
\prod_{i<j}
\frac{(qz_i z_j;t)_{\infty}}{(z_i z_j;t)_{\infty}}
\\
\times
\frac{(ab;q)_{\infty}}{(t;t)_{\infty}(qt;q,t)_{\infty}(ab;q,t)_{\infty}(ac;q,t)_{\infty}(ad;q,t)_{\infty}(bc;q,t)_{\infty}(bd;q,t)_{\infty}(cd;q,t)_{\infty}}.
\end{multline*}
It remains to note that what we wish to compute is given by acting with the involution $\omega^{(N)}$ on $\Lambda(z;q,t;a,b,c,d)$; this is seen by direct comparison of \eqref{partition-function} and \eqref{iterate}. One easily shows that
\begin{align*}
&
\omega^{(N)} : \prod_{i \geq 1} (-az_i;t)_{\infty}
\mapsto
\prod_{i=1}^{N} \frac{1}{(a x_i;q,t)}_{\infty},
\\
\\
&
\omega^{(N)} : 
\prod_{i \geq 1} \frac{1}{(z_i^2;t)_{\infty}}
\prod_{i<j}
\frac{(qz_i z_j;t)_{\infty}}{(z_i z_j;t)_{\infty}}
\mapsto
\prod_{1 \leq i <j \leq N}
\frac{1}{(x_i x_j;q,t)_{\infty}},
\end{align*}
and \eqref{partition-function} then follows immediately.
\end{proof}

\begin{rmk}
Recall the definition \eqref{conjecture} of $Z_{n}(x;q,t;a,b,c,d)$. Sending $n \rightarrow \infty$ and using the fact that
\begin{align*}
h_{\infty}(ab;q)
=
\sum_{k=0}^{\infty}
\frac{(ab)^k}{(q;q)_k}
=
\frac{1}{(ab;q)_{\infty}},
\end{align*}
one has
\begin{align*}
Z_{\infty,N}(q,t;a,b,c,d)
=
\frac{1}
{(ab;q)_{\infty} (q;q)_{\infty}}
\sum_{\mu \subseteq \lambda}
\frac{h_{\lambda'}(a,b;q)}
{\prod_{i \geq 1}(q;q)_{\lambda_i-\lambda_{i+1}}}
\cdot
P_{\lambda/\mu}(x;q,0)
\cdot
t^{|\mu|/2}
h_{\mu'}(c/\sqrt{t},d/\sqrt{t};q).
\end{align*}
Invoking the result of Theorem \ref{thm:pf}, we have the direct evaluation
\begin{multline*}
Z_{\infty,N}(q,t;a,b,c,d)
=
\prod_{i=1}^{N}
\frac{1}{(ax_i;q,t)_{\infty}(bx_i;q,t)_{\infty}(cx_i;q,t)_{\infty}(dx_i;q,t)_{\infty}}
\prod_{1 \leq i<j \leq N}
\frac{1}{(x_i x_j;q,t)_{\infty}}
\\
\times
\frac{1}{(q;q)_{\infty}(t;t)_{\infty}(qt;q,t)_{\infty}(ab;q,t)_{\infty}(ac;q,t)_{\infty}(ad;q,t)_{\infty}(bc;q,t)_{\infty}(bd;q,t)_{\infty}(cd;q,t)_{\infty}},
\end{multline*}
which is manifestly symmetric in $(q,t)$ and $(a,b,c,d)$.
\end{rmk}

\section{Additional symmetries}
\label{sec: extra sym}
In this section, we show that the parameters $(a,b,c,d)$ are equivalent to variables in $x$. This will immediately imply the symmetry in $(a,b,c,d)$ of Theorem \ref{thm: qt symmetry}, and simplify the proof of $(q,t)$ symmetry. We also show the type $D$ symmetry of the variables, as the argument uses very similar ideas.

\subsection{Equivalence of parameters and variables}
\begin{prop}
\label{prop: abcd equiv}
Fix $n\geq 0$ and an alphabet $x=(x_1,\dotsc, x_N)$. Let $(x,a,b,c,d)$ denote the combined alphabet $(a,b,c,d,x_1,\dotsc, x_N)$. We have
\begin{equation*}
    Z_{n}(x;q,t;a,b,c,d)=Z_{n}((x,a,b,c,d);q,t;0,0,0,0).
\end{equation*}
\end{prop}

We will separately show the equivalence for the parameters $(a,b)$ and $(c,d)$. We first state a lemma due to Warnaar which will imply the $(c,d)$ equivalence. It follows from Theorem 4.1 of \cite{W06} after noting that $Q'_{\lambda/\mu}[-\{a,b\};q]=P_{\lambda'/\mu'}(a,b;q,0)$, where $Q'_{\lambda/\mu}$ denotes a modified Hall--Littlewood polynomial. 

\begin{lem}
\label{lem:cd equiv}
We have
\begin{equation*}
    h_{\lambda'}(a,b;q)=\sum_{\mu:\mu'\text{ even}}P_{\lambda/\mu}(a,b;q,0).
\end{equation*}
\end{lem}

Next, we show the following identity which will give the $(a,b)$ equivalence.

\begin{lem}
\label{lem:ab equiv}
Let $a,b$ be parameters and let $\mu$ be a fixed partition. If $\mu_1\leq n$, we have
\begin{equation*}
    \sum_{\lambda:\lambda_1\leq n,\lambda' \text{ even}}\frac{1}{(q;q)_{n-\lambda_1}}Q_{\lambda/\mu}(a,b;q,0)=\sum_{\lambda:\lambda_1\leq n}\frac{h_{\lambda'}(a,0;q)}{(q;q)_{n-\lambda_1}}Q_{\lambda/\mu}(b;q,0)=\frac{h_{n-\mu_1}(ab;q)h_{\mu'}(a,b;q)}{(q;q)_{n-\mu_1}},
\end{equation*}
and otherwise the first two expressions equal $0$.
\end{lem}
\begin{proof}
We first prove the second equality. Note that $h_{\lambda'}(a,0;q)=a^{odd(\lambda')}$, where $odd(\lambda')=\sum_{i}m_{2i-1}(\lambda)$ is the number of odd columns in $\lambda$. We write the left hand side as
\begin{equation*}
\begin{split}
    &\sum_{\lambda_1\leq n}\frac{a^{odd(\lambda')}}{(q;q)_{n-\mu_1}}\prod_{i\geq 0}\frac{(q;q)_{\mu_i-\mu_{i+1}}}{(q;q)_{\lambda_i-\mu_{i+1}}(q;q)_{\mu_i-\lambda_i}}b^{|\lambda/\mu|}
    \\=&\frac{a^{odd(\mu')}}{(q;q)_{n-\mu_1}}\sum_{0\leq m_i\leq \mu_i-\mu_{i+1}}\prod_{\substack{i\geq 0\\i\text{ even}}}\frac{(q;q)_{\mu_i-\mu_{i+1}}}{(q;q)_{\mu_i-\mu_{i+1}-m_i}(q;q)_{m_i}}(ab)^{m_i}\prod_{\substack{i\geq 0\\i\text{ odd}}}\frac{(q;q)_{\mu_i-\mu_{i+1}}}{(q;q)_{\mu_i-\mu_{i+1}-m_i}(q;q)_{m_i}}\left(\frac{b}{a}\right)^{m_i},
\end{split}
\end{equation*}
where we let $\mu_0=n$ and $m_i=\lambda_{i+1}-\mu_{i+1}$. In particular, $\sum_i m_{2i}-m_{2i+1}=odd(\lambda')-odd(\mu')$ and $\sum_i m_i=|\lambda/\mu|$ so the powers in $a$ and $b$ match. We can then distribute the sums, in the even and odd cases obtaining either $h_{\mu_i-\mu_{i+1}}(ab;q)$ or $h_{\mu_i-\mu_{i+1}}(b/a;q)$ respectively. In particular, we have an $i=0$ factor, which by definition gives $h_{n-\mu_1}(ab;q)$, and we recognize the right hand side. If $\mu_1>n$, then the sum is clearly $0$.

The first equality follows from the second. In particular, we have
\begin{equation*}
    \sum_{\lambda:\lambda_1\leq n,\lambda' \text{ even}}Q_{\lambda/\mu}(a,b;q,0)=\sum_{\lambda:\lambda_1\leq n,\lambda' \text{ even}}\sum_{\nu}Q_{\lambda/\nu}(a;q,0)Q_{\nu/\mu}(b;q,0),
\end{equation*}
and the sum over $\lambda$ can be evaluated with the second equality when $a=0$, obtaining
\begin{equation*}
    \sum_{\lambda:\lambda_1\leq n,\lambda' \text{ even}}Q_{\lambda/\nu}(a;q,0)=\frac{h_{\nu'}(a,0;q)}{(q;q)_{n-\nu_1}}
\end{equation*}
if $\nu_1\leq n$, and $0$ otherwise. Thus,
\begin{equation*}
    \sum_{\lambda:\lambda_1\leq n,\lambda' \text{ even}}Q_{\lambda/\mu}(a,b;q,0)=\sum_{\nu:\nu_1\leq n}\frac{h_{\nu'}(a,0;q)}{(q;q)_{n-\nu_1}}Q_{\nu/\mu}(b;q,0),
\end{equation*}
proving the first equality.
\end{proof}

Proposition \ref{prop: abcd equiv} follows easily from Lemmas \ref{lem:ab equiv} and \ref{lem:cd equiv}, and the branching rule.

\begin{proof}[Proof of Proposition \ref{prop: abcd equiv}]
We will separately show that the parameters $(a,b)$ and $(c,d)$ can be absorbed into the variables $x$, which we can do by the branching rule. In more detail, we have that $Z_n(x;q,t;a,b,c,d)$ equals
\begin{equation*}
\begin{split}
&\sum_{\mu \subseteq  \lambda}\sum_{\substack{\nu\subseteq \mu:\\\nu'\text{ even}}}\sum_{\substack{\kappa\supseteq \lambda:\\ \kappa_1\leq n, \kappa'\text{ even}}}
\frac{1}{(q;q)_{n-\kappa_1}\prod_{i \geq 1}(q;q)_{\lambda_i-\lambda_{i+1}}}
\cdot
Q_{\kappa/\lambda}(a,b;q,0)P_{\lambda/\mu}(x;q,0)P_{\mu/\nu}(c,d;q,0)
\cdot
t^{|\nu|/2},
\end{split}
\end{equation*}
using Lemmas \ref{lem:ab equiv} and \ref{lem:cd equiv}. In particular the restriction on $\lambda_1$ is no longer needed. We can then compute the sum over $\mu$ easily using the branching rule. After then using the $\prod_{i\geq 1}(q;q)_{\lambda_i-\lambda_{i+1}}$ factors to write $P_{\lambda/\nu}((x,c,d);q,0)$ as $Q_{\lambda/\nu}((x,c,d);q,0)$, we can use the branching rule to evaluate the sum over $\lambda$, obtaining $Z_n((x,a,b,c,d);q,t;0,0,0,0)$.
\end{proof}

\subsection{Type $D$ symmetry of the variables}
We now show that (after suitable normalization), $Z_n$ exhibits type $D$ symmetry in the variables $x$, in that an even number may be inverted without changing the expression. This follows immediately from the following lemma.

\begin{lem}
\label{lem:inv sym identity}
For all $\mu$, we have
\begin{equation*}
    (xy)^n\sum_{\lambda:\lambda_1\leq n,\lambda\text{ even}}\frac{1}{(q;q)_{n-\lambda_1}}Q_{\lambda/\mu}(x^{-1},y^{-1};q,0)=\sum_{\lambda:\lambda_1\leq n,\lambda\text{ even}}\frac{1}{(q;q)_{n-\lambda_1}}Q_{\lambda/\mu}(x,y;q,0).
\end{equation*}
\end{lem}
\begin{proof}
By Lemma \ref{lem:ab equiv}, we may equivalently prove
\begin{equation*}
    (xy)^n\sum_{\lambda:\lambda_1\leq n}\frac{x^{-odd(\lambda')}}{(q;q)_{n-\lambda_1}}Q_{\lambda/\mu}(y^{-1};q,0)=\sum_{\lambda:\lambda_1\leq n}\frac{x^{odd(\lambda')}}{(q;q)_{n-\lambda_1}}Q_{\lambda/\mu}(y;q,0),
\end{equation*}
which can be written explicitly as
\begin{equation*}
\begin{split}
    &(xy)^n\sum_{\lambda:\lambda_1\leq n}\frac{x^{-odd(\lambda')}y^{-|\lambda/\mu|}}{(q;q)_{n-\lambda_1}}\prod_{i}\frac{(q;q)_{\mu_i-\mu_{i+1}}}{(q;q)_{\lambda_i-\mu_{i}}(q;q)_{\mu_i-\lambda_{i+1}}}
    \\=&\sum_{\lambda:\lambda_1\leq n}\frac{x^{odd(\lambda')}y^{|\lambda/\mu|}}{(q;q)_{n-\lambda_1}}\prod_{i}\frac{(q;q)_{\mu_i-\mu_{i+1}}}{(q;q)_{\lambda_i-\mu_{i}}(q;q)_{\mu_i-\lambda_{i+1}}}.
\end{split}
\end{equation*}
We will prove this bijectively. In particular, the map $\lambda\mapsto \widetilde{\lambda}$ with $\widetilde{\lambda}_i=\mu_i+(\mu_{i-1}-\lambda_i)$ (and $\mu_0=n$ for convenience) is a weight-preserving bijection between terms on either side.

To see this, note that we may reparametrize the sum over $\lambda_i-\mu_i=m_i\in [0,\mu_{i-1}-\mu_i]$, and the map $\lambda\mapsto \widetilde{\lambda}$ complements $m_i$ in each interval. This clearly preserves the product
\begin{equation*}
    (q;q)_{n-\lambda_1}\prod_{i}(q;q)_{\lambda_i-\mu_i}(q;q)_{\mu_i-\lambda_{i+1}},
\end{equation*}
which is nothing more than $\prod_{l\in L} (q;q)_l$ where $L$ is the set of all lengths $m_i=\lambda_i-\mu_i$, and their complements $\widetilde{\lambda}_i-\mu_i$. 

Moreover, one can check that $|\widetilde{\lambda}/\mu|=n-|\lambda/\mu|$ and $n-odd(\widetilde{\lambda}')=odd(\lambda')$ using the fact that the map $\lambda\mapsto \widetilde{\lambda}$ has this interpretation in terms of complementation. In particular, there are $n$ total boxes in the intervals $[\mu_i,\mu_{i+1}]$ of which $|\lambda/\mu|$ are used, and so after complementation $|\widetilde{\lambda}/\mu|$ are used. Similarly, within each interval $[\mu_i,\mu_{i+1}]$, if $i$ is even, there are $\lambda_{i+1}-\mu_{i+1}$ many odd columns and $\mu_{i}-\lambda_{i+1}$ many even columns, and vice versa if $i$ is odd. In either case, $\lambda\mapsto \widetilde{\lambda}$ exchanges the even and odd columns in each interval.
\end{proof}

\begin{lem}
\label{lem: inv sym}
Fix a finite alphabet $x=(x_1,\dotsc, x_N)$. We have
\begin{equation*}
    (x_1x_2)^nZ_n(x_1^{-1},x_2^{-2},x_3,\dotsc, x_N;q,t;a,b,c,d)=Z_n(x;q,t;a,b,c,d).
\end{equation*}
\end{lem}
\begin{proof}
By Proposition \ref{prop: abcd equiv}, we can take $a=b=c=d=0$. Then Lemma \ref{lem:inv sym identity} and the branching rule immediately imply the equality.
\end{proof}

The following proposition is immediate from Lemma \ref{lem: inv sym}, Proposition \ref{prop: abcd equiv}, and the fact that $Z_n$ is symmetric in the $x$ variables.
\begin{prop}
Fix a finite alphabet $x=(x_1,\dotsc, x_N)$. 
The expression
\begin{equation*}
    \prod_{i}x_i^{-n/2}Z_n(x;q,t;a,b,c,d)
\end{equation*}
is invariant under an even number of transformations $x_i\mapsto x_i^{-1}$. In fact, the expression 
\begin{equation*}
    (abcd)^{-n/2}\prod_{i}x_i^{-n/2}Z_n(x;q,t;a,b,c,d)
\end{equation*}
is invariant under the simultaneous transformations $x_i\mapsto x_i^{-1}$, $a\mapsto a^{-1}$.
\end{prop}

\section{Contour integral formulas for free boundary $q$-Whittaker measure}
\label{sec: main pf}
We note that by Proposition \ref{prop: abcd equiv}, we can take $a=b=c=d=0$ in the proof of Theorem \ref{thm: qt symmetry} without loss of generality. For convenience, we define $Z_n(x;q,t)=Z_n(x;q,t;0,0,0,0)$. We begin by stating the main result of this section, which are the following contour integral formulas for $Z_n(x;q,t)$ which immediately imply the $(q,t)$ symmetry of Theorem \ref{thm: qt symmetry}. To ease the notation, we introduce the convention that $f(y_i^{\pm})=f(y_i)f(y_i^{-1})$ for any function $f$. So for example $(1-y_i^{\pm}y_j^{\pm})=(1-y_iy_j)(1-y_iy_j^{-1})(1-y_i^{-1}y_j)(1-y_i^{-1}y_j^{-1})$. If $y=(y_1,\dotsc, y_m)$, we let
\begin{equation}
    \Delta(q,t;y)=\prod_{i<j}\frac{(1-y_i^{\pm}y_j^{\pm})(1-qty_i^{\pm}y_j^{\pm})}{(1-qy_i^{\pm}y_j^{\pm})(1-ty_i^{\pm}y_j^{\pm})}\prod_{i}\frac{1}{(1-y_i^{\pm 2})(1-qy_i^{\pm 2})(1-ty_i^{\pm 2})}.
\end{equation}

\begin{thm}
\label{thm: contour int nice}
Fix $q,t<1$, and let $C$ be a contour of radius $r$ centered at $0$, for some $r$ such that $1<r<\min(q^{-1/2},t^{-1/2})$. Let $x=(x_1,\dotsc, x_N)$ be a finite alphabet. Then if $n=2m$, we have
\begin{equation*}
\begin{split}
Z_{2m}(x;q,t)=&\frac{1}{2^mm!}\frac{(1-qt)^m}{(1-q)^m(1-t)^m}\oint_{C}\frac{dy_1}{2\pi i}\dotsm \oint_{C}\frac{dy_m}{2\pi i}\Delta(q,t;y) \prod_{i,j}(1+x_iy_j^{\pm})
\\&+\frac{1}{2^{m-1}(m-1)!}\frac{(1-qt)^m(1+qt)}{2(1-q)^{m}(1-q^2)(1-t)^{m}(1-t^2)}\oint_{C}\frac{dy_1}{2\pi i}\dotsm \oint_{C}\frac{dy_{m-1}}{2\pi i}\Delta(q,t;y)
\\&\qquad\qquad\times\prod_{i}\frac{(1-y_i^{\pm 2})(1-q^2t^2y_i^{\pm 2})}{(1-q^2y_i^{\pm 2})(1-t^2y_i^{\pm 2})}\prod_{i,j}(1+x_iy_j^{\pm})\prod_i (1-x_i^2)
\end{split}
\end{equation*}
and if $n=2m+1$, we have
\begin{multline*}
Z_{2m+1}(x;q,t)=\frac{1}{2^mm!}\frac{(1-qt)^m}{(1-q)^{m+1}(1-t)^{m+1}}\oint_{C}\frac{dy_1}{2\pi i}\dotsm \oint_{C}\frac{dy_m}{2\pi i} \Delta(q,t;y)\prod_{i,j}(1+x_iy_j^{\pm})
\\ \times\left[\prod_i \frac{(1-y_i^{\pm})(1-qty_i^{\pm})}{(1-qy_i^{\pm})(1-ty_i^{\pm})}\prod_j (1+x_j)+\prod_i \frac{(1+y_i^{\pm})(1+qty_i^{\pm})}{(1+qy_i^{\pm})(1+ty_i^{\pm})}\prod_j (1-x_j)\right].
\end{multline*}
\end{thm}

\begin{rmk}
One can also interpret the contour integrals in Theorem \ref{thm: contour int nice} formally. This involves expanding the integrand into a power series in $q$ and $t$, whose coefficients are Laurent polynomials in the integration variables. The contour integral then simply picks out the constant term.
\end{rmk}

\begin{rmk}
Although the formulas given in Theorem \ref{thm: contour int nice} appear unwieldy, let us point out that the integrands have many remarkable properties. In particular, they have a symmetry in the parameters $(q,t)$ which imply the $(q,t)$ symmetry of Theorem \ref{thm: qt symmetry}. Moreover, they also have a BC symmetry in the integration variables $y_i$ (they can be exchanged and inverted without affecting the integrand). Together with the fact that $2^mm!$ is the size of the type BC Weyl group, this perhaps suggests that our formulas are symmetrized versions of a simpler formula over the type BC Weyl group. 
\end{rmk}

\begin{rmk}
In order to add the parameters $a,b,c,d$ back into the formulas, one simply adds $a,b,c,d$ to the variables $x_i$. This will result in an extra factor of $\prod_i (1+ay_i^{\pm})(1+by_i^{\pm})(1+cy_i^{\pm})(1+dy_i^{\pm})$ (along with some additional constant factors in the appropriate places). One particularly nice simplification occurs if $a=1$ and $b=-1$, as then $Z_{2m+1}=0$ and the second integral in $Z_{2m}$ vanishes, leading to a simpler formula. 
\end{rmk}

Let us give an outline of the proof of Theorem \ref{thm: contour int nice}. In Section \ref{sec: prelim formula}, we prove a simpler contour integral formula, Proposition \ref{prop: simpler formula}, which does not have the desired symmetry. We then analyze the residues in Section \ref{sec: residues}, and show that Theorem \ref{thm: contour int nice} follows after evaluating some (but not all) of the contour integrals in terms of residues.

\subsection{A preliminary formula}
\label{sec: prelim formula}
In order to prove Theorem \ref{thm: contour int nice}, we first establish a simpler contour integral formula, which unfortunately does not have any of the desired symmetry properties.
\begin{prop}
\label{prop: simpler formula}
We have 
\begin{multline}
\label{eq: simpler formula}
Z_{n}(x;q,t)
=
\frac{1}{{n!(1-q)^n}}
\oint_{C} \frac{y_1dy_1}{2\pi i }
\cdots
\oint_{C} \frac{y_ndy_n}{2\pi i}
\\
\prod_{1 \leq i\not= j \leq n}
\frac{y_i-y_j}{y_i-q y_j}
\prod_i\frac{1}{(y_i^2-1)(1-ty_i^2)}\prod_{1\leq i<j\leq n}\frac{(y_iy_j-q)(1-qty_iy_j)}{(y_iy_j-1)(1-ty_iy_j)}
\prod_{i=1}^{n}
\prod_{j=1}^{N}
(1+y_i x_j),
\end{multline}
where now the contours $C$ are circles of radius $r$, with $1<r<t^{-1/2}$.
\end{prop}
The proof of this proposition will occupy the rest of this subsection.

\subsubsection{Integral formula for Hall--Littlewood functions}

Our first step is to invoke a contour integral formula for skew Hall--Littlewood functions. This integral formula is a standard fact that may be derived directly from the skew Cauchy identity and orthogonality of Macdonald polynomials; see \cite[Chapter VI, Sections 4, 7 and 9]{M79}. Here we present this result as it was stated in \cite{HW23}:

\begin{prop}
Let $z=(z_1,z_2,\dots)$ be an alphabet of arbitrary size. Fix partitions $\mu \subseteq \lambda$ and an integer $n \geq 1$ such that $\ell(\lambda) \leq n$. The skew Hall--Littlewood function $P_{\lambda/\mu}(z;0,q)$ is given by the $n$-fold contour integral
\begin{multline}
\label{integral-hall}
P_{\lambda/\mu}(z;0,q)
=
\frac{(q;q)_{n-\ell(\lambda)}}{n!(1-q)^n}
\oint_{C} \frac{dy_1}{2\pi i y_1}
\cdots
\oint_{C} \frac{dy_n}{2\pi i y_n}
\\
\times
\prod_{1 \leq i\not= j \leq n}
\frac{y_i-y_j}{y_i-q y_j}
P_{\lambda}(y;0,q)
Q_{\mu}(y^{-1};0,q)
\prod_{i=1}^{n}
\prod_{j \geq 1}
\frac{y_i-qz_j}{y_i-z_j},
\end{multline}
where $C$ is a positively oriented circle centered on the origin that encloses the points $(z_1,z_2,\dots)$.
\end{prop}

\subsubsection{Macdonald involution}

We wish to convert \eqref{integral-hall} into an integral formula for skew Whittaker polynomials. Similarly to the work of \cite{HW23}, it is convenient for us to consider the image \eqref{integral-hall} under the Macdonald involution \cite[Chapter VI, Section 5]{M79}. In particular, we define a map 
\begin{align*}
\omega^{(N)}: 
\Lambda[z_1,z_2,\dots] \rightarrow \Lambda[x_1,\dots,x_N],
\end{align*}
where $\Lambda[z_1,z_2,\dots]$ denotes the ring of symmetric functions in an infinite alphabet $(z_1,z_2,\dots)$ and $\Lambda[x_1,\dots,x_N]$ is its finite version on $(x_1,\dots,x_N)$, with action defined on power sums:
\begin{align*}
\omega^{(N)} : p_k(z) \mapsto \frac{(-1)^{k-1}}{1-q^k} p_k(x).
\end{align*}
This map acts on skew Hall--Littlewood functions as follows:
\begin{align}
\label{involution-act}
\omega^{(N)}:
Q_{\lambda/\mu}(z;0,q)
\mapsto
P_{\lambda'/\mu'}(x;q,0),
\end{align}
where the right hand side is a skew $q$-Whittaker polynomial indexed by conjugated partitions, and 
\begin{align}
\label{PQ-HL}
Q_{\lambda/\mu}(x;0,q)
=
\frac{b_{\lambda}(0,q)}{b_{\mu}(0,q)}
P_{\lambda/\mu}(x;0,q),
\qquad
b_{\lambda}(0,q) = \prod_{i \geq 1} (q;q)_{m_i(\lambda)}.
\end{align}
The involution also has a nice action on the final product appearing in the integrand of \eqref{integral-hall}:
\begin{align}
\label{cauchy-involution}
\omega^{(N)}:
\prod_{i=1}^{n}
\prod_{j \geq 1}
\frac{y_i-q z_j }{y_i-z_j}
\mapsto
\prod_{i=1}^{n}
\prod_{j=1}^{N}
(1+y_i^{-1} x_j).
\end{align}

\begin{prop}
Fix integers $n,N \geq 1$, an alphabet $x=(x_1,\dots,x_N)$ and partitions $\mu \subseteq \lambda$ such that $\lambda_1 \leq n$.
The skew $q$-Whittaker polynomial $P_{\lambda/\mu}(x;q,0)$ is given by
\begin{multline}
\label{qW-integral}
P_{\lambda/\mu}(x;q,0)
=
\frac{(q;q)_{n-\lambda_1}}{n!(1-q)^n}
\oint_{C} \frac{dy_1}{2\pi i y_1}
\cdots
\oint_{C} \frac{dy_n}{2\pi i y_n}
\\
\times
\prod_{1 \leq i\not= j \leq n}
\frac{y_i-y_j}{y_i-q y_j}
Q_{\lambda'}(y;0,q)
P_{\mu'}(y^{-1};0,q)
\prod_{i=1}^{n}
\prod_{j=1}^{N}
(1+y_i^{-1} x_j)
\end{multline}
where $C$ is a positively oriented circle centered on the origin.
\end{prop}

\begin{proof}
After multiplying \eqref{integral-hall} by $b_{\lambda}(0,q)/b_{\mu}(0,q)$, this follows by application of \eqref{involution-act} and \eqref{cauchy-involution} to its left and right hand sides, combined with relabelling partitions $\lambda \mapsto \lambda'$, $\mu \mapsto \mu'$.

To argue that $\omega^{(N)}$ passes through the integrals, we note the integrals can be interpreted formally by writing the integrand as a power series in $q$ and $t$, with coefficients Laurent polynomials in the $y_i$ and symmetric functions in the $x_i$. The contour integrals then amount to taking constant terms, which clearly commute with $\omega^{(N)}$.
\end{proof}

\subsubsection{Proof of formula}
Our next step is to cast the expression $Z_n(x;q,t)$ as a contour integral. Using the fact that
\begin{align*}
Q_{\lambda'}(y;0,q)
=
b_{\lambda'}(0,q)
P_{\lambda'}(y;0,q),
\qquad
b_{\lambda'}(0,q)
=
\prod_{i \geq 1}(q;q)_{\lambda_i-\lambda_{i+1}},
\end{align*}
we may substitute the integral \eqref{qW-integral} directly into \eqref{conjecture}, which yields
\begin{multline*}
Z_{n}(x;q,t)
=
\frac{1}{{n!(1-q)^n}}
\sum_{\substack{\mu \subseteq \lambda:\lambda_1 \leq n,\\\mu',\lambda'\text{ even}}}
t^{|\mu|/2}
\\
\times
\oint_{C} \frac{dy_1}{2\pi i y_1}
\cdots
\oint_{C} \frac{dy_n}{2\pi i y_n}
\prod_{1 \leq i\not= j \leq n}
\frac{y_i-y_j}{y_i-q y_j}
P_{\lambda'}(y;0,q)
P_{\mu'}(y^{-1};0,q)
\prod_{i=1}^{n}
\prod_{j=1}^{N}
(1+y_i^{-1} x_j).
\end{multline*}
Finally, one may make the relabeling $\lambda \mapsto \lambda'$, $\mu \mapsto \mu'$ everywhere in the previous expression, and we obtain
\begin{multline}
\label{integral-with-sums}
Z_{n}(x;q,t)
=
\frac{1}{{n!(1-q)^n}}
\sum_{\substack{\mu \subseteq \lambda:\ell(\lambda) \leq n\\\mu,\lambda\text{ even}}}
t^{|\mu|/2}
\\
\times
\oint_{C} \frac{dy_1}{2\pi i y_1}
\cdots
\oint_{C} \frac{dy_n}{2\pi i y_n}
\prod_{1 \leq i\not= j \leq n}
\frac{y_i-y_j}{y_i-q y_j}
P_{\lambda}(y;0,q)
P_{\mu}(y^{-1};0,q)
\prod_{i=1}^{n}
\prod_{j=1}^{N}
(1+y_i^{-1} x_j).
\end{multline}
Note that we may now remove the restriction that $\ell(\lambda)\leq n$ as $P_\lambda(y;0,q)=0$ if this is violated. Finally, we use the Littlewood identity \eqref{littlewood} to complete the proof.

\subsection{Analysis of residues}
\label{sec: residues}
We now show how Proposition \ref{prop: simpler formula} can be used to derive Theorem \ref{thm: contour int nice}. Let
\begin{multline*}
    I_n(y;q,t)=
    \\\frac{1}{(1-q)^n}\prod_{1 \leq i\not= j \leq n}
\frac{y_i-y_j}{y_i-q y_j}
\prod_i\frac{y_i}{(y_i^2-1)(1-ty_i^2)}\prod_{1\leq i<j\leq n}\frac{(y_iy_j-q)(1-qty_iy_j)}{(y_iy_j-1)(1-ty_iy_j)}
\prod_{i=1}^{n}
\prod_{j=1}^{N}
(1+y_i x_j),
\end{multline*}
so that $Z_n(x;q,t)=\frac{1}{n!}\oint_{C} \frac{dy_1}{2\pi i }
\cdots
\oint_{C} \frac{dy_n}{2\pi i}I_n(y;q,t)$. 

\subsubsection{Informal description of inductive procedure}
We begin by informally describing an inductive procedure to evaluate the contour integrals. We will verify all claims below afterwards. 

Note $y_n$ has three types of poles that lie within $C$, of the form
\begin{enumerate}[(a)]
    \item $y_n=qy_i$ for some $i<n$,
    \item $y_n=\pm 1$,
    \item $y_n=y_i^{-1}$ for some $i<n$.
\end{enumerate}
We claim that the residues of type (a) evaluate to $0$, and so may be ignored. By symmetry of the $y_i$, the residues of type $(c)$ evaluate to the same expression, so we may take $i=n-1$. Thus, we obtain a sum of $(n-1)$-fold contour integrals via the residue theorem. 

The first key observation is that the integrand for the $+1$ residue no longer has a $+1$ residue for any of the $y_i$, $i<n$, and similarly for the $-1$ residue. The second key observation is that the resulting integrands of all the residues will have the same types of residues for all variables except $y_{n-1}$ in type (c), which we would consider fixed, and the same restriction on residues of type (b). We may thus continue this process. The end result is that we would inductively consider residues of type (c), except that we may pick a $\pm 1$ residue of type (b) exactly once for each sign. We obtain either an $n/2$-, $(n-1)/2$-, or $(n-2)/2$-fold contour integral, which will give us Theorem \ref{thm: contour int nice}. We will now make this informal description precise.

\subsubsection{Classification of residues}
We note that in Proposition \ref{prop: simpler formula} we may choose $C$ to be the circle of radius $r$ centered at $0$, for some $r$ such that $1<r<\min(q^{-1/2},t^{-1/2})$, which we now do for the rest of the proof. We first give the pole structure of $I_n(y;q,t)$. The following lemma is easy to verify.

\begin{lem}
The poles of $I_n(y;q,t)$ in the variable $y_n$ that lie within $C$ are of the form
\begin{enumerate}[(a)]
    \item $y_n=qy_i$ for some $i<n$,
    \item $y_n=\pm 1$,
    \item $y_n=y_i^{-1}$ for some $i<n$.
\end{enumerate}
Moreover, $\res_{y_n=qy_i}I_n(y;q,t)=0$.
\end{lem}

We let $\mathcal{M}$ denote the set of tuples $(z_1,\dotsc, z_k)$, where the $z_i$ are all distinct, and either $z_i=\pm 1$, or $z_i=y_j^{-1}$ for some $j$ such that $j$ is smaller than $n-i$ plus the number of $i<j$ for which $z_i=y_l^{-1}$ for $l>j$. To any data $m\in\mathcal{M}$, we can compute an iterated residue, by first computing the residue of $y_n$ at $z_1$, and then given we have reached $z_i$, computing the residue of $y_j$ at $z_i$, where $j$ equals $n-i$ plus the number of $i<j$ for which $z_i=y_l^{-1}$ for $l>j$ (i.e. the largest $j$ which has not been used up in this process), and we use $\res_m$ to denote this operation. We let $y\setminus m$ denote the set of variables $y_i$ which do not appear in $m$ and whose residues have not yet been computed (so the result of $\res_m$ is a function of $y\setminus m$, plus the variables appearing in $m$). We write $y_i\in m$ (or $\pm 1\in m$) if $z_j=y_i^{-1}$ (or $z_j=\pm 1$) for some $j$.

\begin{lem}
\label{lem: induction}
Let $F(y;q,t)$ be a rational function satisfying the following properties in the variables $y$:
\begin{enumerate}
    \item It is symmetric in the variables $y$,
    \item The poles of $F$ in the variables $y_i$ are independent of $y_j$ for $j\neq i$, and lie outside the contour $C$.
\end{enumerate}
Then for all $m\in \mathcal{M}$, the expression
\begin{equation*}
    F_m=\frac{\res_m F(y;q,t)I_n(y;q,t)}{I_{n-|y\setminus m|}(y\setminus m;q,t)}
\end{equation*}
satisfies the same two properties for variables $y\setminus m$:
\begin{enumerate}
    \item It is symmetric in the variables $y\setminus m$,
    \item The poles of $F_m$ in the variables $y_i\in y\setminus m$ are independent of $y_j\in y\setminus m$ for $j\neq i$, and lie outside the contour $C$ when $y_i\in C$ for all $y_i\in m$,
\end{enumerate}
and moreover $\res_m F(y;q,t)I_n(y;q,t)$ has no poles at $y_i=1$ if $1\in m$, and no poles at $y_i=-1$ if $-1\in m$.
\end{lem}
\begin{proof}
We induct on the length of $m$ and $n$. If $|m|=0$, then this is clear. Assume now the statement for all tuples of length at most $k$ and all $n'<n$, and take $m=(z_1,m')$ of length $k+1$. Note that we may swap the order of residues by symmetry (or the fact that they will always be simple poles), so we will first compute the residue of $y_n=z_1$.

In case (b) when $z_1=\pm 1$, we compute
\begin{equation*}
    F_+=\frac{\res_{y_n=1}(I_n(y;q,t))}{I_{n-1}(y\setminus (z_1);q,t)}=\frac{1}{2(1-q)(1-t)}\prod_{1\leq i<n}\frac{(1-y_i)}{(1-qy_i)}\frac{(1-qty_i)}{(1-ty_i)}\prod_{j=1}^N(1+x_j)
\end{equation*}
and
\begin{equation*}
    F_-=\frac{\res_{y_n=-1}(I_n(y;q,t))}{I_{n-1}(y\setminus (z_1);q,t)}=\frac{1}{2(1-q)(1-t)}\prod_{1\leq i<n}\frac{(1+y_i)}{(1+qy_i)}\frac{(1+qty_i)}{(1+ty_i)}\prod_{j=1}^N(1-x_j).
\end{equation*}
It is easily seen that these rational functions satisfy properties (1) and (2) (for the variables $y_1,\dotsc, y_{n-1}$). Then by induction, as 
\begin{equation*}
    F_m=\frac{\res_{m'}F_{\pm} I_{n-1}(y\setminus(z_1);q,t)}{I_{n-|y\setminus m|}(y\setminus m;q,t)},
\end{equation*}
we have that $F_m$ also satisfies properties (1) and (2). Moreover, it's clear by property (2) that the order of a pole $y_i=\pm 1$ in $\res_m F(y;q,t)I_n(y;q,t)$ is of order at most $1$, coming from $I_n(y;q,t)$, but if $z_1=\pm 1$ (or any $z_i$ by symmetry), then a factor of $\prod_{1\leq i<n}(1\mp y_i)$ will be introduced, removing the poles.

Finally, in case (c), we first note that there is no harm in assuming $i=n-1$, since by assumption we have symmetry in the $y_i$. We then compute
\begin{equation*}
\begin{split}
    &F_{y_{n-1}}=\frac{\res_{y_n=y_{n-1}^{-1}}(I_n(y;q,t))}{I_{n-2}(y\setminus (y_{n-1});q,t)}
    \\=&\frac{(1-qt)}{(1-q)(1-t)}\frac{y_{n-1}^4}{(y_{n-1}^2-q)(1-qy_{n-1}^2)(1-ty_{n-1}^2)(y_{n-1}^2-t)}
    \\&\times\prod_{1\leq i<n-1}\frac{(y_i-y_{n-1})(y_iy_{n-1}-1)}{(y_{n-1}-qy_i)(1-qy_iy_{n-1})}\frac{(1-qty_iy_{n-1})(y_{n-1}-qty_i)}{(1-ty_iy_{n-1})(y_{n-1}-ty_i)}\prod_{j=1}^N(1+y_{n-1}x_j)(1+y_{n-1}^{-1}x_j).
\end{split}
\end{equation*}
For fixed $y_{n-1}\in C$, this rational function satisfies properties (1) and (2) (and in particular we do not care that $y_{n-1}$ has poles in $C$). Then by induction, as
\begin{equation*}
    F_m=\frac{\res_{m'}F_{y_{n-1}} I_{n-2}(y\setminus(y_{n-1});q,t)}{I_{n-|y\setminus m|}(y\setminus m;q,t)},
\end{equation*}
we have that $F_m$ also satisfies properties (1) and (2).
\end{proof}

\subsubsection{Proof of Theorem \ref{thm: contour int nice}}
The inductive procedure given above gives a method to compute the integral (in terms of a contour integral with fewer variables), by choosing a residue of type (b) or (c), starting with $y_n$, and if type (c) then also removing from consideration the other variable chosen. Since all poles considered are simple, the process of taking residues in different variables commutes (assuming that if we pick $y_n=y_i^{-1}$, we never consider a residue of $y_i$). By Lemma \ref{lem: induction}, no new poles will be introduced, and if we ever pick a pole of type $(b)$, we remove that type of pole for the remaining variables.

We immediately see that
\begin{equation*}
    \frac{1}{n!}\oint_{C} \frac{dy_1}{2\pi i }
\cdots
\oint_{C} \frac{dy_n}{2\pi i}I_n(y;q,t)=\frac{1}{n!}\sum_{m\in \overline{\mathcal{M}}}\oint_{C}
\cdots
\oint_{C} \prod_{y\in m}\frac{dy}{2\pi i}\res_m I_n(y;q,t),
\end{equation*}
where $\overline{\mathcal{M}}\subseteq \mathcal{M}$ denotes the set of maximal tuples (i.e. ones which cannot be extended). By symmetry, each term only depends on whether $\pm 1$ appear in $m$, and not on any other choices made.

If $n=2m+1$, we must pick a residue of the form $y_i=\pm 1$ exactly once, and the remaining choices amounts to a perfect matching on the remaining variables. By symmetry, this choice doesn't matter, and so we can compute the integral as
\begin{equation*}
Z_{2m+1}(x;q,t)=\frac{1}{2^m m!}\oint_{C}\frac{dy_1}{2\pi i}\dotsm \oint_{C}\frac{dy_m}{2\pi i}I_{2m+1}(y,y^{-1},1;q,t)+I_{2m+1}(y,y^{-1},-1;q,t),
\end{equation*}
where $(y,y^{-1},1)$ denotes the specialization $(y_1,\dotsc, y_m,y_1^{-1},\dotsc, y_m^{-1},1)$, and similarly for $(y,y^{-1},-1)$. Here, we gain a factor of $(2m+1)\frac{2m!}{2^mm!}$ for the number of ways to choose the residues. It is not hard to check that
\begin{multline*}
I_{2m+1}(y^{\pm},\pm 1;q,t)=\frac{(1-qt)^m}{2(1-q)^{m+1}(1-t)^{m+1}} \prod_{i<j}\frac{(1-y_i^{\pm}y_j^{\pm})(1-qty_i^{\pm}y_j^{\pm})}{(1-qy_i^{\pm}y_j^{\pm})(1-ty_i^{\pm}y_j^{\pm})}
\\ \times \prod_{i}\frac{1}{(1-y_i^{\pm 2})(1-qy_i^{\pm 2})(1-ty_i^{\pm 2})}\prod_{i,j}(1+x_iy_j^{\pm})\prod_i \frac{(1-y_i^{\pm})(1-qty_i^{\pm})}{(1-qy_i^{\pm})(1-ty_i^{\pm})}\prod_j (1\pm x_j).
\end{multline*}

If $n=2m$, we can either pick both $y_i=1$ and $y_i=-1$ as a residue at some point, or never pick either. In either case, the other choices amount to a perfect matching on the remaining variables, and again by symmetry these choices don't matter, except whether we pick the type (b) residues or not. Thus, we have
\begin{multline*}
Z_{2m}(x;q,t)=\frac{1}{2^mm!}\oint_{C}\frac{dy_1}{2\pi i}\dotsm \oint_{C}\frac{dy_m}{2\pi i}I_{2m}(y,y^{-1};q,t)
\\+\frac{1}{2^{m-1}(m-1)!}\oint_{C}\frac{dy_1}{2\pi i}\dotsm \oint_{C}\frac{dy_{m-1}}{2\pi i}I_{2m}(y,y^{-1},1,-1;q,t).
\end{multline*}
Here, we either get a factor of $\frac{(2m)!}{2^mm!}$ or $\frac{(2m)!}{2^{m-1}(m-1)!}$, depending on whether we pick $\pm1$ or not, as the remaining choice is again just a perfect matching. Again, we can compute
\begin{multline*}
I_{2m}(y,y^{-1};q,t)=\frac{(1-qt)^m}{(1-q)^m(1-t)^m}\prod_{i<j}\frac{(1-y_i^{\pm}y_j^{\pm})(1-qty_i^{\pm}y_j^{\pm})}{(1-qy_i^{\pm}y_j^{\pm})(1-ty_i^{\pm}y_j^{\pm})}
\\\times\prod_{i}\frac{1}{(1-y_i^{\pm 2})(1-qy_i^{\pm 2})(1-ty_i^{\pm 2})}
 \prod_{i,j}(1+x_iy_j^{\pm}),
\end{multline*}
and
\begin{multline*}
I_{2m}(y,y^{-1},1,-1;q,t)=\frac{(1-qt)^m(1+qt)}{2(1-q)^{m+1}(1+q)(1-t)^{m+1}(1+t)}\prod_{i<j}\frac{(1-y_i^{\pm}y_j^{\pm})(1-qty_i^{\pm}y_j^{\pm})}{(1-qy_i^{\pm}y_j^{\pm})(1-ty_i^{\pm}y_j^{\pm})}
\\\times\prod_{i}\frac{(1-q^2t^2y_i^{\pm 2})}{(1-qy_i^{\pm 2})(1-q^2y_i^{\pm 2})(1-ty_i^{\pm 2})(1-t^2y_i^{\pm 2})}\prod_{i,j}(1+x_iy_j^{\pm})\prod_i (1-x_i^2).
\end{multline*}
This completes the proof of Theorem \ref{thm: contour int nice}.

\section{Vertex models and Hall--Littlewood processes}
\label{sec: HL process}
In this section, we first introduce the quasi-open six vertex model, and then show a distributional equality relating certain observables there to the free boundary Hall--Littlewood process, after a suitable random shift.

\subsection{Stochastic six-vertex model}
\subsubsection{Vertex weights}

We introduce the basic notions of the stochastic six-vertex model, following \cite[Chapter 2]{BW21}.

A {\it fundamental vertex} is the intersection of an oriented horizontal and vertical line. Two complex parameters, called {\it rapidities}, are associated to a vertex; one of these is associated to the horizontal line and the other to the vertical line. A discrete variable in the set $\{0,1\}$ is assigned to each of the four line segments attached to the vertex. Given such data, we associate a weight function to vertices as follows:
\begin{align}
\label{R-vert}
R_{y/x}(i,j; k,\ell)
=
\begin{tikzpicture}[scale=0.7,baseline=-0.1cm]
\node[left] at (-1.5,0) {$x$};
\node[below] at (0,-1.5) {$y$};
\draw[lgray,line width=1.5pt,->] (-1,0) -- (1,0);
\draw[lgray,line width=1.5pt,->] (0,-1) -- (0,1);
\node[left] at (-1,0) {\tiny $j$};\node[right] at (1,0) {\tiny $\ell$};
\node[below] at (0,-1) {\tiny $i$};\node[above] at (0,1) {\tiny $k$};
\end{tikzpicture}
\quad\quad\quad
i,j,k,\ell \in \{0,1\},
\end{align}
where $R_{y/x}(i,j;k,\ell)$ is a rational function in the rapidities $x,y$ of the vertex, that depends only on their ratio.

One can interpret the above figure as the propagation of lattice paths through a vertex: each edge variable equal to 1 represents a path superimposed over that edge, while an edge variable equal to 0 indicates that no path is present. The {\it incoming} paths are those situated at the left and bottom edges of the vertex; those at the right and top are called {\it outgoing}. The weight of the vertex, $R_{y/x}(i,j; k,\ell)$, is chosen to vanish identically unless the total flux of paths through the vertex is preserved. This gives rise to six possible non-trivial vertices, whose weights we tabulate below:
\begin{align}
\label{fund-vert}
\begin{tabular}{|c|c|c|}
\hline
\quad
\tikz{0.6}{
	\draw[lgray,line width=1.5pt,->] (-1,0) -- (1,0);
	\draw[lgray,line width=1.5pt,->] (0,-1) -- (0,1);
	\node[left] at (-1,0) {\tiny $0$};\node[right] at (1,0) {\tiny $0$};
	\node[below] at (0,-1) {\tiny $0$};\node[above] at (0,1) {\tiny $0$};
}
\quad
&
\quad
\tikz{0.6}{
	\draw[lgray,line width=1.5pt,->] (-1,0) -- (1,0);
	\draw[lgray,line width=1.5pt,->] (0,-1) -- (0,1);
	\node[left] at (-1,0) {\tiny $0$};\node[right] at (1,0) {\tiny $0$};
	\node[below] at (0,-1) {\tiny $1$};\node[above] at (0,1) {\tiny $1$};
	\draw[ultra thick,->] (0,-1) -- (0,1);
}
\quad
&
\quad
\tikz{0.6}{
	\draw[lgray,line width=1.5pt,->] (-1,0) -- (1,0);
	\draw[lgray,line width=1.5pt,->] (0,-1) -- (0,1);
	\node[left] at (-1,0) {\tiny $0$};\node[right] at (1,0) {\tiny $1$};
	\node[below] at (0,-1) {\tiny $1$};\node[above] at (0,1) {\tiny $0$};
	\draw[ultra thick,->] (0,-1) --(0,0)-- (1,0);
}
\quad
\\[1.3cm]
\quad
$1$
\quad
& 
\quad
$\dfrac{q(1-y/x)}{1-qy/x}$
\quad
& 
\quad
$\dfrac{1-q}{1-qy/x}$
\quad
\\[0.7cm]
\hline
\quad
\tikz{0.6}{
	\draw[lgray,line width=1.5pt,->] (-1,0) -- (1,0);
	\draw[lgray,line width=1.5pt,->] (0,-1) -- (0,1);
	\node[left] at (-1,0) {\tiny $1$};\node[right] at (1,0) {\tiny $1$};
	\node[below] at (0,-1) {\tiny $1$};\node[above] at (0,1) {\tiny $1$};
	\draw[ultra thick,->] (0,-1) -- (0,1);
	\draw[ultra thick,->] (-1,0) -- (1,0);
}
\quad
&
\quad
\tikz{0.6}{
	\draw[lgray,line width=1.5pt,->] (-1,0) -- (1,0);
	\draw[lgray,line width=1.5pt,->] (0,-1) -- (0,1);
	\node[left] at (-1,0) {\tiny $1$};\node[right] at (1,0) {\tiny $1$};
	\node[below] at (0,-1) {\tiny $0$};\node[above] at (0,1) {\tiny $0$};
	\draw[ultra thick,->] (-1,0) -- (1,0);
}
\quad
&
\quad
\tikz{0.6}{
	\draw[lgray,line width=1.5pt,->] (-1,0) -- (1,0);
	\draw[lgray,line width=1.5pt,->] (0,-1) -- (0,1);
	\node[left] at (-1,0) {\tiny $1$};\node[right] at (1,0) {\tiny $0$};
	\node[below] at (0,-1) {\tiny $0$};\node[above] at (0,1) {\tiny $1$};
	\draw[ultra thick,->] (-1,0) --(0,0)-- (0,1);
}
\quad
\\[1.3cm]
\quad
$1$
\quad
& 
\quad
$\dfrac{1-y/x}{1-qy/x}$
\quad
&
\quad
$\dfrac{(1-q)y/x}{1-qy/x}$
\quad 
\\[0.7cm]
\hline
\end{tabular}
\end{align}

\begin{prop}[Yang--Baxter equation]
For any fixed vector $(i_1,i_2,i_3,j_1,j_2,j_3) \in \{0,1\}^6$, the following identity of weights holds:
\begin{multline}
\label{component-YB}
\sum_{0 \leq k_1,k_2,k_3 \leq 1}
R_{y/x}(i_2,i_1;k_2,k_1) 
R_{z/x}(i_3,k_1;k_3,j_1) 
R_{z/y}(k_3,k_2;j_3,j_2)
\\
=
\sum_{0 \leq k_1,k_2,k_3 \leq 1}
R_{z/y}(i_3,i_2;k_3,k_2) 
R_{z/x}(k_3,i_1;j_3,k_1) 
R_{y/x}(k_2,k_1;j_2,j_1).
\end{multline}
\end{prop}

\begin{proof}
The $2^6$ equations can be directly verified using the explicit form \eqref{fund-vert} of the vertex weights.
\end{proof}
The Yang--Baxter equation \eqref{component-YB} has the pictorial representation
\begin{align*}
\sum_{0 \leq k_1,k_2,k_3 \leq 1}
\tikz{0.8}{
\draw[lgray,line width=1.5pt,->]
(-2,1) node[above,scale=0.6] {\color{black} $i_1$} -- (-1,0) node[below,scale=0.6] {\color{black} $k_1$} -- (1,0) node[right,scale=0.6] {\color{black} $j_1$};
\draw[lgray,line width=1.5pt,->] 
(-2,0) node[below,scale=0.6] {\color{black} $i_2$} -- (-1,1) node[above,scale=0.6] {\color{black} $k_2$} -- (1,1) node[right,scale=0.6] {\color{black} $j_2$};
\draw[lgray,line width=1.5pt,->] 
(0,-1) node[below,scale=0.6] {\color{black} $i_3$} -- (0,0.5) node[scale=0.6] {\color{black} $k_3$} -- (0,2) node[above,scale=0.6] {\color{black} $j_3$};
\node[left] at (-2.2,1) {$x$};
\node[left] at (-2.2,0) {$y$};
\node[below] at (0,-1.5) {$z$};
}
\quad
=
\quad
\sum_{0 \leq k_1,k_2,k_3 \leq 1}
\tikz{0.8}{
\draw[lgray,line width=1.5pt,->] 
(-1,1) node[left,scale=0.6] {\color{black} $i_1$} -- (1,1) node[above,scale=0.6] {\color{black} $k_1$} -- (2,0) node[below,scale=0.6] {\color{black} $j_1$};
\draw[lgray,line width=1.5pt,->] 
(-1,0) node[left,scale=0.6] {\color{black} $i_2$} -- (1,0) node[below,scale=0.6] {\color{black} $k_2$} -- (2,1) node[above,scale=0.6] {\color{black} $j_2$};
\draw[lgray,line width=1.5pt,->] 
(0,-1) node[below,scale=0.6] {\color{black} $i_3$} -- (0,0.5) node[scale=0.6] {\color{black} $k_3$} -- (0,2) node[above,scale=0.6] {\color{black} $j_3$};
\node[left] at (-1.5,1) {$x$};
\node[left] at (-1.5,0) {$y$};
\node[below] at (0,-1.5) {$z$};
}
\end{align*}
which is usually its most practical form in calculations.

\subsubsection{Boundary vertices}

Following \cite{GGMW23}, one may also introduce boundary vertices and write down their corresponding reflection equation. We fix parameters $a,b$ for the first boundary, and parameters $c,d$ for the other boundary.

A {\it fundamental boundary vertex} is a node with two oriented edges attached to it; one edge is again incoming, and the other is outgoing. Rapidities are associated to both the incoming and outgoing edges, but they must be equal up to reciprocation. A discrete variable in the set $\{0,1\}$ is assigned to each of the two line segments attached to the node. Given such data, we associate a weight function to boundary vertices as follows:
\begin{align}
\label{K-vert}
K_x(i;j)
=
\begin{tikzpicture}[scale=0.7,baseline=0.5cm]
\draw[lightgray,line width=1.5pt,->] (1,0) -- (2,1) -- (1,2);	\draw[black,fill=black] (2,1) circle (0.1cm);
\node[left] at (0.3,0) {$x^{-1}$};
\node[left] at (0.3,2) {$x$};
\node[left] at (1,0) {$i$};
\node[left] at (1,2) {$j$};
\end{tikzpicture}
\quad\quad\quad
i,j \in \{0,1\},
\end{align}
where $K_x(i;j)$ is a rational function in $x$. 

Unlike in the case of the vertices \eqref{R-vert}, we impose no conservation constraint on the boundary vertex \eqref{K-vert}, meaning that lattice paths may begin or terminate at the node. As such, there are four non-trivial boundary vertex weight. We let $h_1(x)=\frac{1-x^2}{(1+ax)(1+bx)}$ and $h_2(x)=\frac{1-x^2}{(1+cx)(1+dx)}$. Then the vertex weights are given by:

\begin{align}
\label{Stochastic boundary K-weight table}
\begin{tabular}{|c|c|c|c|}
\hline
\qquad
\begin{tikzpicture}[scale=0.7,baseline=0.5cm]
\draw[lightgray,line width=1.5pt,->] (1,0) -- (2,1) -- (1,2);	\draw[black,fill=black] (2,1) circle (0.1cm);
\node[left] at (1,0) {$0$};
\node[left] at (1,2) {$0$};
\end{tikzpicture}
\qquad
&
\qquad
\begin{tikzpicture}[scale=0.7,baseline=0.5cm]
\draw[lightgray,line width=1.5pt,->] (1,0) -- (2,1) -- (1,2);	\draw[black,fill=black] (2,1) circle (0.1cm);
\draw[ultra thick,->] (2,1) -- (1,2);
\node[left] at (1,0) {$0$};
\node[left] at (1,2) {$1$};
\end{tikzpicture}
\qquad
&
\qquad
\begin{tikzpicture}[scale=0.7,baseline=0.5cm]
\draw[lightgray,line width=1.5pt,->] (1,0) -- (2,1) -- (1,2);	\draw[black,fill=black] (2,1) circle (0.1cm);
\draw[ultra thick] (1,0) -- (2,1);
\node[left] at (1,0) {$1$};
\node[left] at (1,2) {$0$};
\end{tikzpicture}
\qquad
&
\qquad
\begin{tikzpicture}[scale=0.7,baseline=0.5cm]
\draw[lightgray,line width=1.5pt,->] (1,0) -- (2,1) -- (1,2);	\draw[black,fill=black] (2,1) circle (0.1cm);
\draw[ultra thick,->] (1,0)--(2,1) -- (1,2);
\node[left] at (1,0) {$1$};
\node[left] at (1,2) {$1$};
\end{tikzpicture}
\\[0.7cm]
\hline
\qquad
$1+abh_1(x)$
\qquad
&
\qquad
$-abh_1(x)$
\qquad
&
\qquad
$ h_1(x)$
\qquad
&
\qquad
$1-h_1(x)$
\qquad
\\[0.7cm]
\hline
\end{tabular}
\end{align}

We will also need the dual weights for the second boundary, which we will denote with a red dot. These weights are actually nothing more than the weights given by \eqref{Stochastic boundary K-weight table} with $x$ and the parameters $a,b$ inverted, but we give them explicitly below:

\begin{align}
\label{Stochastic boundary dual K-weight table}
\begin{tabular}{|c|c|c|c|}
\hline
\qquad
\begin{tikzpicture}[scale=0.7,baseline=0.5cm]
\draw[lightgray,line width=1.5pt,->] (1,0) -- (0,1) -- (1,2);	
\node[right] at (1,0) {$0$};
\node[right] at (1,2) {$0$};
\draw[red,fill=red] (0,1) circle (0.1cm);
\end{tikzpicture}
\qquad
&
\qquad
\begin{tikzpicture}[scale=0.7,baseline=0.5cm]
\draw[lightgray,line width=1.5pt,->] (1,0) -- (0,1) -- (1,2);	
\draw[ultra thick,->] (0,1) -- (1,2);
\node[right] at (1,0) {$0$};
\node[right] at (1,2) {$1$};
\draw[red,fill=red] (0,1) circle (0.1cm);
\end{tikzpicture}
\qquad
&
\qquad
\begin{tikzpicture}[scale=0.7,baseline=0.5cm]
\draw[lightgray,line width=1.5pt,->] (1,0) -- (0,1) -- (1,2);	
\draw[ultra thick] (1,0) -- (0,1);
\node[right] at (1,0) {$1$};
\node[right] at (1,2) {$0$};
\draw[red,fill=red] (0,1) circle (0.1cm);
\end{tikzpicture}
\qquad
&
\qquad
\begin{tikzpicture}[scale=0.7,baseline=0.5cm]
\draw[lightgray,line width=1.5pt,->] (1,0) -- (0,1) -- (1,2);	
\draw[ultra thick,->] (1,0)--(0,1) -- (1,2);
\node[right] at (1,0) {$1$};
\node[right] at (1,2) {$1$};
\draw[red,fill=red] (0,1) circle (0.1cm);
\end{tikzpicture}
\\[0.7cm]
\hline
\qquad
$1-h_2(x)$
\qquad
&
\qquad
$h_2(x)$
\qquad
&
\qquad
$-cd h_2(x)$
\qquad
&
\qquad
$1+cd h_2(x)$
\qquad
\\[0.7cm]
\hline
\end{tabular}
\end{align}

%
%

\subsubsection{Quasi-open six vertex model}
\label{sec: quasi-open 6vm}
Let $N\in\N$, let $t\in [0,1)$, and let $x=(x_1,\dotsc, x_N)$ be parameters in $(0,1)$ which we call rapidities. The stochastic six vertex model on a triangle of size $N$ is a probability measure on configurations of arrows in the upper left half of an $N\times N$ lattice. All vertices have degree $4$, except the vertices on the diagonal which have degree $2$. Each edge can be occupied by at most one arrow, and arrows travel upwards and rightwards. There are two versions, depending on whether we take the upper or lower half of the $N\times N$ lattice, and we use boundary weights \eqref{Stochastic boundary K-weight table} and \eqref{Stochastic boundary dual K-weight table} respectively. More generally, we will allow an extra parameter $z$, and we will write

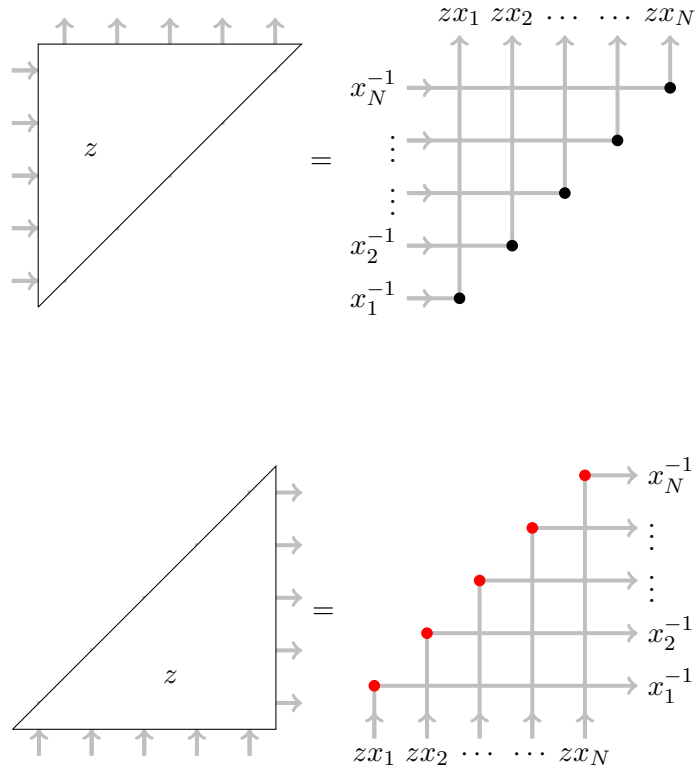
\begin{equation*}
\begin{tikzpicture}[baseline={([yshift=-.5ex]current bounding box.center)},scale=0.7]
			\draw[lightgray,line width=1.5pt,->] (0,1) -- (1,1) -- (1,6);
			\draw[lightgray,line width=1.5pt,->] (0,2) -- (2,2) -- (2,6);
			\draw[lightgray,line width=1.5pt,->] (0,3) -- (3,3) -- (3,6);
			\draw[lightgray,line width=1.5pt,->] (0,4) -- (4,4) -- (4,6);
			\draw[lightgray,line width=1.5pt,->] (0,5) -- (5,5) -- (5,6);
			
			\draw[lightgray,line width=1.5pt,->] (0,1) -- (0.5,1);
			\draw[lightgray,line width=1.5pt,->] (0,2) -- (0.5,2);
			\draw[lightgray,line width=1.5pt,->] (0,3) -- (0.5,3);
			\draw[lightgray,line width=1.5pt,->] (0,4) -- (0.5,4);
			\draw[lightgray,line width=1.5pt,->] (0,5) -- (0.5,5);
    \draw[fill=white](0.5,0.5)--(0.5,5.5)--(5.5,5.5)--(0.5,0.5);\draw[](1.5,3.5)node[]{$z$};
\end{tikzpicture}
=
    \begin{tikzpicture}[baseline={([yshift=-.5ex]current bounding box.center)},scale=0.7]
			\draw[lightgray,line width=1.5pt,->] (0,1) -- (1,1) -- (1,6);
			\draw[lightgray,line width=1.5pt,->] (0,2) -- (2,2) -- (2,6);
			\draw[lightgray,line width=1.5pt,->] (0,3) -- (3,3) -- (3,6);
			\draw[lightgray,line width=1.5pt,->] (0,4) -- (4,4) -- (4,6);
			\draw[lightgray,line width=1.5pt,->] (0,5) -- (5,5) -- (5,6);
			
			\draw[lightgray,line width=1.5pt,->] (0,1) -- (0.5,1);
			\draw[lightgray,line width=1.5pt,->] (0,2) -- (0.5,2);
			\draw[lightgray,line width=1.5pt,->] (0,3) -- (0.5,3);
			\draw[lightgray,line width=1.5pt,->] (0,4) -- (0.5,4);
			\draw[lightgray,line width=1.5pt,->] (0,5) -- (0.5,5);
			
			\draw[black,fill=black] (5,5) circle (0.1cm);
			\draw[black,fill=black] (4,4) circle (0.1cm);
			\draw[black,fill=black] (3,3) circle (0.1cm);
			\draw[black,fill=black] (2,2) circle (0.1cm);
			\draw[black,fill=black] (1,1) circle (0.1cm);
			
			\node[above] at (1,6) {$zx_1$};
			\node[above] at (2,6) {$zx_2$};
			\node[above] at (3,6) {$\cdots$};
			\node[above] at (4,6) {$\cdots$};
			\node[above] at (5,6) {$zx_N$};

			\node[left] at (0,1) {$x_1^{-1}$};
			\node[left] at (0,2) {$x_2^{-1}$};
			\node[left] at (0,3) {$\vdots$};
			\node[left] at (0,4) {$\vdots$};
			\node[left] at (0,5) {$x_N^{-1}$};
		\end{tikzpicture}
\end{equation*}

\begin{equation*}
\begin{tikzpicture}[baseline={([yshift=-.5ex]current bounding box.center)},scale=0.7]
			\draw[lightgray,line width=1.5pt,->] (1,0) -- (1,1) -- (6,1);
			\draw[lightgray,line width=1.5pt,->] (2,0) -- (2,2) -- (6,2);
			\draw[lightgray,line width=1.5pt,->] (3,0) -- (3,3) -- (6,3);
			\draw[lightgray,line width=1.5pt,->] (4,0) -- (4,4) -- (6,4);
			\draw[lightgray,line width=1.5pt,->] (5,0) -- (5,5) -- (6,5);
			
			\draw[lightgray,line width=1.5pt,->] (1,0) -- (1,0.5);
			\draw[lightgray,line width=1.5pt,->] (2,0) -- (2,0.5);
			\draw[lightgray,line width=1.5pt,->] (3,0) -- (3,0.5);
			\draw[lightgray,line width=1.5pt,->] (4,0) -- (4,0.5);
			\draw[lightgray,line width=1.5pt,->] (5,0) -- (5,0.5);
    \draw[fill=white](0.5,0.5)--(5.5,0.5)--(5.5,5.5)--(0.5,0.5);\draw[](3.5,1.5)node[]{$z$};
\end{tikzpicture}
=
    \begin{tikzpicture}[baseline={([yshift=-.5ex]current bounding box.center)},scale=0.7]
			\draw[lightgray,line width=1.5pt,->] (1,0) -- (1,1) -- (6,1);
			\draw[lightgray,line width=1.5pt,->] (2,0) -- (2,2) -- (6,2);
			\draw[lightgray,line width=1.5pt,->] (3,0) -- (3,3) -- (6,3);
			\draw[lightgray,line width=1.5pt,->] (4,0) -- (4,4) -- (6,4);
			\draw[lightgray,line width=1.5pt,->] (5,0) -- (5,5) -- (6,5);
			
			\draw[lightgray,line width=1.5pt,->] (1,0) -- (1,0.5);
			\draw[lightgray,line width=1.5pt,->] (2,0) -- (2,0.5);
			\draw[lightgray,line width=1.5pt,->] (3,0) -- (3,0.5);
			\draw[lightgray,line width=1.5pt,->] (4,0) -- (4,0.5);
			\draw[lightgray,line width=1.5pt,->] (5,0) -- (5,0.5);
			
			\draw[red,fill=red] (5,5) circle (0.1cm);
			\draw[red,fill=red] (4,4) circle (0.1cm);
			\draw[red,fill=red] (3,3) circle (0.1cm);
			\draw[red,fill=red] (2,2) circle (0.1cm);
			\draw[red,fill=red] (1,1) circle (0.1cm);
			
			\node[right] at (6,1) {$x_1^{-1}$};
			\node[right] at (6,2) {$ x_2^{-1}$};
			\node[right] at (6,3) {$\vdots$};
			\node[right] at (6,4) {$\vdots$};
			\node[right] at (6,5) {$x_N^{-1}$};

			\node[below] at (1,0) {$zx_1$};
			\node[below] at (2,0) {$zx_2$};
			\node[below] at (3,0) {$\cdots$};
			\node[below] at (4,0) {$\cdots$};
			\node[below] at (5,0) {$zx_N$};
		\end{tikzpicture}
\end{equation*}
where we will specify the boundary conditions for arrows entering/exiting the triangle.

We then define the \emph{quasi-open stochastic six vertex model of length $2L$} to be the stochastic six vertex model represented by
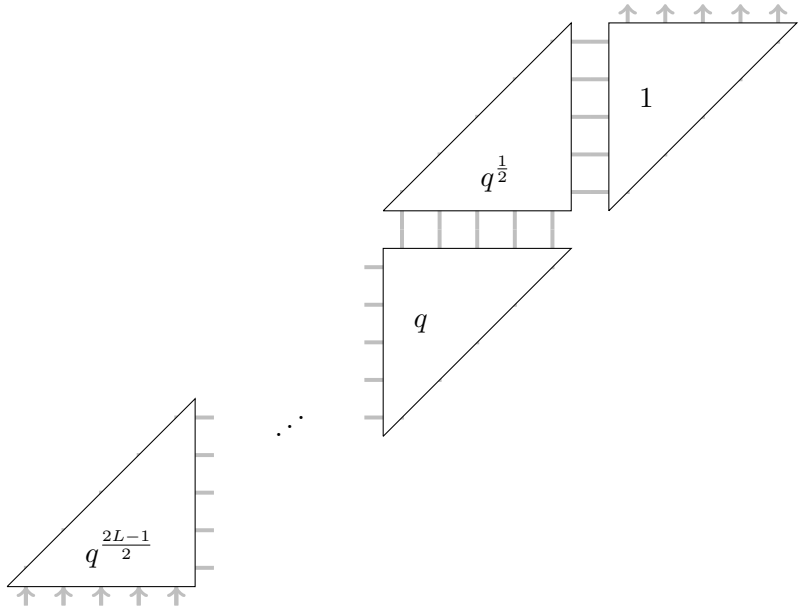
\begin{equation*}
    \begin{tikzpicture}[scale=0.5]
        			\draw[lightgray,line width=1.5pt,->] (0,1) -- (1,1) -- (1,6);
			\draw[lightgray,line width=1.5pt,->] (0,2) -- (2,2) -- (2,6);
			\draw[lightgray,line width=1.5pt,->] (0,3) -- (3,3) -- (3,6);
			\draw[lightgray,line width=1.5pt,->] (0,4) -- (4,4) -- (4,6);
			\draw[lightgray,line width=1.5pt,->] (0,5) -- (5,5) -- (5,6);
                \draw[fill=white](0.5,0.5)--(0.5,5.5)--(5.5,5.5)--(0.5,0.5);\draw[](1.5,3.5)node[]{$1$};
            \begin{scope}[shift=({-6,0})]
            \draw[lightgray,line width=1.5pt] (1,0) -- (1,1) -- (6,1);
			\draw[lightgray,line width=1.5pt] (2,0) -- (2,2) -- (6,2);
			\draw[lightgray,line width=1.5pt] (3,0) -- (3,3) -- (6,3);
			\draw[lightgray,line width=1.5pt] (4,0) -- (4,4) -- (6,4);
			\draw[lightgray,line width=1.5pt] (5,0) -- (5,5) -- (6,5);
                \draw[fill=white](0.5,0.5)--(5.5,0.5)--(5.5,5.5)--(0.5,0.5);\draw[](3.5,1.5)node[]{$q^{\frac{1}{2}}$};
            \end{scope}
            \begin{scope}[shift=({-6,-6})]
                \draw[lightgray,line width=1.5pt] (0,1) -- (1,1) -- (1,6);
			\draw[lightgray,line width=1.5pt] (0,2) -- (2,2) -- (2,6);
			\draw[lightgray,line width=1.5pt] (0,3) -- (3,3) -- (3,6);
			\draw[lightgray,line width=1.5pt] (0,4) -- (4,4) -- (4,6);
			\draw[lightgray,line width=1.5pt] (0,5) -- (5,5) -- (5,6);
                \draw[fill=white](0.5,0.5)--(0.5,5.5)--(5.5,5.5)--(0.5,0.5);\draw[](1.5,3.5)node[]{$q$};

            \end{scope}
            \draw[](-8,-5)node[]{$\iddots$};
                        \begin{scope}[shift=({-16,-10})]
            \draw[lightgray,line width=1.5pt] (1,0) -- (1,1) -- (6,1);
			\draw[lightgray,line width=1.5pt] (2,0) -- (2,2) -- (6,2);
			\draw[lightgray,line width=1.5pt] (3,0) -- (3,3) -- (6,3);
			\draw[lightgray,line width=1.5pt] (4,0) -- (4,4) -- (6,4);
			\draw[lightgray,line width=1.5pt] (5,0) -- (5,5) -- (6,5);
            			\draw[lightgray,line width=1.5pt,->] (1,0) -- (1,0.5);
			\draw[lightgray,line width=1.5pt,->] (2,0) -- (2,0.5);
			\draw[lightgray,line width=1.5pt,->] (3,0) -- (3,0.5);
			\draw[lightgray,line width=1.5pt,->] (4,0) -- (4,0.5);
			\draw[lightgray,line width=1.5pt,->] (5,0) -- (5,0.5);
                \draw[fill=white](0.5,0.5)--(5.5,0.5)--(5.5,5.5)--(0.5,0.5);\draw[](3.5,1.5)node[]{$q^{\frac{2L-1}{2}}$};
            \end{scope}
    \end{tikzpicture}
\end{equation*}
where there are $2L$ triangles of alternating types, the boundary condition on the bottom is empty, and we view the arrows exiting the top as random variables. 

Note that the limit as $L\to\infty$ is well-defined as long as $q\in (0,1)$, since the geometric decay ensures that almost surely, all but a finite portion of the model consists of arrows traveling in a horizontal direction. Thus, we define the \emph{quasi-open six vertex model} to be the $L\to\infty$ limit of the quasi-open six vertex model of length $2L$. The following theorem states that certain observables match those of the free boundary Hall--Littlewood process (which may not always be a probability measure). For a sequence of partitions $\vec\lambda=(\lambda^{(0),\dotsc, \lambda^{(N)}})$, the \emph{support} of $\vec\lambda$ is $[\vec\lambda]=(s_1,\dotsc, s_N)$, where $s_i=I_{l(\lambda^{(i)})=l(\lambda^{(i+1)})+1}$.

\begin{thm}
\label{thm: HL 6vm}
Let $\chi$ be a random variable (or possibly a signed measure) with probability generating function given by
\begin{equation*}
G(z)=\frac{(abcdz^2;q,t)_\infty}{(qz;q)_\infty (tz;t)_\infty (qtz;q,t)_\infty(abcdz;q,t)_\infty }\prod_{s\in \left\{\substack{ab,ac,ad,\\bc,bd,cd}\right\}}\frac{1}{(sz;q,t)_\infty}.
\end{equation*}
Let $S$ denote the locations of outgoing arrows in the quasi-open six vertex model, and let $H$ and $V$ denote the total numbers of horizontal unoccupied edges and vertical occupied edges between triangles. Let $\vec\lambda\sim \mathbb{FBHL}_{q,t,a,b,c,d}^{x}$. Then
\begin{equation*}
\mathbb{FBHL}_{q,t,a,b,c,d}^{x}(\ell(\lambda^{(0)})=n,[\vec\lambda]=s)=\mathbb{P}(\chi+H+V=n,S=s).
\end{equation*}
\end{thm}

\begin{rmk}
Note that since $ab,cd\leq 0$, the free boundary Hall--Littlewood measure may be a signed measure in general, and similarly for $\chi$. Thus, we may have to interpret Theorem \ref{thm: HL 6vm} as equality of measures of certain sets, and $\chi+H+V$ in terms of convolution of measures rather than addition of random variables. We will continue to use probabilistic notation even in this case.
\end{rmk}

\begin{rmk}
\label{rmk: HL partition fn}
Note that the right hand side of Theorem \ref{thm: HL 6vm} is easy to compute in the case $n=0$. This gives an independent argument to compute the partition function on the left hand side, $\Phi_{HL}$, since the proof of Theorem \ref{thm: HL 6vm} does not require knowledge of $\Phi_{HL}$.
\end{rmk}

\subsubsection{Relationship to six vertex model on a strip}
The \emph{open stochastic six vertex model} (or six vertex model on a strip) is the $u=1$ case of the quasi-open stochastic six vertex model of length $L$ (for some finite $L$). This is a finite state space Markov chain on $\{0,1\}^N$, and it has a unique stationary measure as long as the boundary is not degenerate, i.e. arrows can enter and/or exit (for example this is ensured if $h(x_i)\in (0,1)$). The following proposition relates the quasi-open stochastic six vertex model with stationary measures for the open six vertex model, see e.g. \cite{Y24}. The proof is identical to that of Proposition 5.9 in \cite{HW23}.

\begin{prop}
Assume that the parameters are chosen such that the open stochastic six vertex model has a unique stationary measure. As $q\to 1$, $S$ converges in distribution to the stationary measure for the open stochastic six vertex model.
\end{prop}

\begin{rmk}
Let us note that at least in a formal sense, one should be able to obtain the two-layer Gibbs measures constructed in \cite{BCY24} from the free boundary $q$-Whittaker process, and we believe that a similar limiting process would allow us to obtain a two-layer Gibbs measure for the six vertex model. In particular, this measure would be obtained by setting $u=1$, removing the restriction that the partitions have positive parts, and restricting them to have two columns. Heuristically, one can see that as $u\to 1$, the partitions which appear will have $m_{2i}(\lambda)$ large ($m_{2i+1}(\lambda)$ is controlled by the $h_\lambda$ factor, which decays as long as $|a|,|b|,|c|,|d|<1$). Asking whether the length increases in the process then becomes a question of just the first two columns. (should we work out more details? I think Ivan/Guillaume are working on this). 
\end{rmk}

\subsection{Deformed bosons}
\subsubsection{Vertex weights}
We now introduce an additional vertex model which is needed in the proof of Theorem \ref{thm: HL 6vm}. It is a model for arrows on a lattice, where the horizontal edges again can have at most one arrow, but where the vertical edges can have any number of arrows. A local parameter (called a \emph{rapidity}) is associated to each row of the lattice, while all vertices depend on a global parameter $t$. The vertex weights are given by

\begin{align}
\label{eq:black-vertices}
\begin{array}{cccc}
\begin{tikzpicture}[scale=0.8,>=stealth]
\draw[lgray,ultra thick] (-1,0) -- (1,0);
\draw[lgray,line width=10pt] (0,-1) -- (0,1);
\node[below] at (0,-1) {$m$};
\draw[ultra thick,->,rounded corners] (-0.075,-1) -- (-0.075,1);
\draw[ultra thick,->,rounded corners] (0.075,-1) -- (0.075,1);
\node[above] at (0,1) {$m$};
\end{tikzpicture}
\quad\quad\quad
&
\begin{tikzpicture}[scale=0.8,>=stealth]
\draw[lgray,ultra thick] (-1,0) -- (1,0);
\draw[lgray,line width=10pt] (0,-1) -- (0,1);
\node[below] at (0,-1) {$m$};
\draw[ultra thick,->,rounded corners] (-0.075,-1) -- (-0.075,1);
\draw[ultra thick,->,rounded corners] (0.075,-1) -- (0.075,0) -- (1,0);
\node[above] at (0,1) {$m-1$};
\end{tikzpicture}
\quad\quad\quad
&
\begin{tikzpicture}[scale=0.8,>=stealth]
\draw[lgray,ultra thick] (-1,0) -- (1,0);
\draw[lgray,line width=10pt] (0,-1) -- (0,1);
\node[below] at (0,-1) {$m$};
\draw[ultra thick,->,rounded corners] (-1,0) -- (-0.15,0) -- (-0.15,1);
\draw[ultra thick,->,rounded corners] (0,-1) -- (0,1);
\draw[ultra thick,->,rounded corners] (0.15,-1) -- (0.15,1);
\node[above] at (0,1) {$m+1$};
\end{tikzpicture}
\quad\quad\quad
&
\begin{tikzpicture}[scale=0.8,>=stealth]
\draw[lgray,ultra thick] (-1,0) -- (1,0);
\draw[lgray,line width=10pt] (0,-1) -- (0,1);
\node[below] at (0,-1) {$m$};
\draw[ultra thick,->,rounded corners] (-1,0) -- (-0.15,0) -- (-0.15,1);
\draw[ultra thick,->,rounded corners] (0,-1) -- (0,1);
\draw[ultra thick,->,rounded corners] (0.15,-1) -- (0.15,0) -- (1,0);
\node[above] at (0,1) {$m$};
\end{tikzpicture}
\\
1
\quad\quad\quad
&
x
\quad\quad\quad
&
(1-t^{m+1})
\quad\quad\quad
&
x
\end{array}\end{align}
where $0\leq x<1$ is the row rapidity, and $m$ is the number of arrows entering from below.   

We will also need another version of this model with an alternative normalization. It is defined in the same way, but with alternative vertex weights

\begin{align}
\label{eq:red-vertices}
\begin{array}{cccc}
\begin{tikzpicture}[scale=0.8,>=stealth]
\draw[lred, ultra thick] (-1,0) -- (1,0);
\draw[lred,line width=10pt] (0,-1) -- (0,1);
\node[below] at (0,-1) {$m$};
\draw[ultra thick,->,rounded corners] (-0.075,-1) -- (-0.075,1);
\draw[ultra thick,->,rounded corners] (0.075,-1) -- (0.075,1);
\node[above] at (0,1) {$m$};
\end{tikzpicture}
\quad\quad\quad
&
\begin{tikzpicture}[scale=0.8,>=stealth]
\draw[lred,ultra thick] (-1,0) -- (1,0);
\draw[lred,line width=10pt] (0,-1) -- (0,1);
\node[below] at (0,-1) {$m$};
\draw[ultra thick,->,rounded corners] (-0.075,-1) -- (-0.075,1);
\draw[ultra thick,->,rounded corners] (0.075,-1) -- (0.075,0) -- (1,0);
\node[above] at (0,1) {$m-1$};
\end{tikzpicture}
\quad\quad\quad
&
\begin{tikzpicture}[scale=0.8,>=stealth]
\draw[lred,ultra thick] (-1,0) -- (1,0);
\draw[lred,line width=10pt] (0,-1) -- (0,1);
\node[below] at (0,-1) {$m$};
\draw[ultra thick,->,rounded corners] (-1,0) -- (-0.15,0) -- (-0.15,1);
\draw[ultra thick,->,rounded corners] (0,-1) -- (0,1);
\draw[ultra thick,->,rounded corners] (0.15,-1) -- (0.15,1);
\node[above] at (0,1) {$m+1$};
\end{tikzpicture}
\quad\quad\quad
&
\begin{tikzpicture}[scale=0.8,>=stealth]
\draw[lred,ultra thick] (-1,0) -- (1,0);
\draw[lred,line width=10pt] (0,-1) -- (0,1);
\node[below] at (0,-1) {$m$};
\draw[ultra thick,->,rounded corners] (-1,0) -- (-0.15,0) -- (-0.15,1);
\draw[ultra thick,->,rounded corners] (0,-1) -- (0,1);
\draw[ultra thick,->,rounded corners] (0.15,-1) -- (0.15,0) -- (1,0);
\node[above] at (0,1) {$m$};
\end{tikzpicture}
\\
y
\quad\quad\quad
&
1
\quad\quad\quad
&
y (1-t^{m+1})
\quad\quad\quad
&
1
\end{array}
\end{align}
where again, $0\leq y<1$ is the row rapidity and $m$ is the number of arrows entering from below.

In both cases, we will use a graphical notation to write partition functions. For a single row, we will write $w_a(\cdot )$ around a picture to represent the sum over all internal edges of the products of the vertex weights, with external edges fixed and $a$ denoting the row rapidity. If there is more than one row, we will instead simply give the picture with row rapidities identified.

We shall also wish to consider infinite rows, formally defined as a limit of progressively longer finite rows. For finitely supported sequences of non-negative integers $m_i$ and $n_i$, we will let
\begin{multline}
\label{eq:L-limit}
    w_x
\left(
\begin{tikzpicture}[baseline=(current bounding box.center),>=stealth,scale=0.8]
\draw[lgray,ultra thick] (0,0) -- (7,0);
\foreach\x in {1,...,6}{
\draw[lgray,line width=10pt] (7-\x,-1) -- (7-\x,1);
}
\node[left] at (0,0) {$0$};
\node[right] at (7,0) {$j$};
\foreach\x in {1,2,3}{
\node[text centered,below] at (7-\x,-1) {\tiny $m_{\x}$};
\node[text centered,above] at (7-\x,1) {\tiny $n_{\x}$};
}
\foreach\x in {4,5}{
\node[text centered,below] at (7-\x,-1) {\tiny $\cdots$};
\node[text centered,above] at (7-\x,1) {\tiny $\cdots$};
}
\end{tikzpicture}\right)
\\=\lim_{N\to\infty}w_x
\left(
\begin{tikzpicture}[baseline=(current bounding box.center),>=stealth,scale=0.8]
\draw[lgray,ultra thick] (0,0) -- (7,0);
\foreach\x in {1,...,6}{
\draw[lgray,line width=10pt] (7-\x,-1) -- (7-\x,1);
}
\node[left] at (0,0) {$0$};
\node[right] at (7,0) {$j$};
\foreach\x in {1,2,3}{
\node[text centered,below] at (7-\x,-1) {\tiny $m_{\x}$};
\node[text centered,above] at (7-\x,1) {\tiny $n_{\x}$};
}
\foreach\x in {4,5}{
\node[text centered,below] at (7-\x,-1) {\tiny $\cdots$};
\node[text centered,above] at (7-\x,1) {\tiny $\cdots$};
}
\node[text centered,below] at (1,-1) {\tiny $m_{N}$};
\node[text centered,above] at (1,1) {\tiny $n_{N}$};
\end{tikzpicture}\right),
\end{multline}
and similarly for the second set of weights \eqref{eq:red-vertices}, except with an arrow entering from the left. These limits are well-defined, because the $m_i$ and $n_i$ are finitely supported, so eventually the weights will all be $1$ sufficiently far to the left in the infinite product of vertices. Note that if an arrow entered from the left in \eqref{eq:L-limit} (or no arrow entered from the left in the case of rows constructed from the weights \eqref{eq:red-vertices}), the weights would eventually all be $a$, and not $1$, and the limit would simply be $0$ since $a<1$.

The reason that this model is useful to prove Theorem \ref{thm: HL 6vm} is that it shares the same $R$-matrix as the six vertex model, but its partition functions are given by Hall--Littlewood polynomials. These two statements are given in the following proposition and lemma. We will use a cross rotated by $45^\circ$ to denote a \emph{Yang--Baxter vertex}; this is a vertex with the same weights as those of the six--vertex model with corresponding row and column parameter.

\begin{prop}[{\hspace{1sp}\cite[Proposition 4.8]{BBCW18}}]
\label{prop: YB boson}
For any finitely supported sequences $n_i$ and $m_i$, and any $j_1,j_2 \in \{0,1\}$, we have
\begin{equation*}
\label{graph-exchange}
\left(
\frac{1-x y}{1-t x y}
\right)
\sum_{p_i}
\begin{tikzpicture}[baseline=(current bounding box.center),>=stealth,scale=0.7]
\draw[lgray,ultra thick] (-1,1) node[left,black] {$x$}
-- (4,1) node[right,black] {$j_2$};
\draw[lred,ultra thick] (-1,0) node[left,black] {$y$}
-- (4,0) node[right,black] {$j_1$};
\foreach\x in {0,...,3}{
\draw[lgray,line width=10pt] (3-\x,0.5) -- (3-\x,2);
\draw[lred,line width=10pt] (3-\x,-1) -- (3-\x,0.5);
}
\node[below] at (3,-1) {$m_1$};
\node at (3,0.5) {$p_1$};
\node[above] at (3,2) {$n_1$};
\node[below] at (2,-1) {$m_2$};
\node at (2,0.5) {$p_2$};
\node[above] at (2,2) {$n_2$};
\node[text centered] at (0,0.5) {$\cdots$};
\node[text centered] at (1,0.5) {$\cdots$};
\draw[ultra thick,->] (-1,0) -- (0,0);
\end{tikzpicture}
=
\sum_{p_i,k_1,k_2}
\begin{tikzpicture}[baseline=(current bounding box.center),>=stealth,scale=0.7]
\foreach\x in {0,...,3}{
\draw[lgray,line width=10pt] (3-\x,-1) -- (3-\x,0.5);
\draw[lred,line width=10pt] (3-\x,0.5) -- (3-\x,2);
}
\draw[lred,ultra thick] (-1,1) node[left,black] {$b$}
-- (4,1);
\draw[lgray,ultra thick] (-1,0) node[left,black] {$a$}
-- (4,0);
\draw[dotted,thick] (4,1) node[above] {$k_1$} -- (5,0) node[right] {$j_1$};
\draw[dotted,thick] (4,0) node[below] {$k_2$} -- (5,1) node[right] {$j_2$};
\node[below] at (3,-1) {$m_1$};
\node at (3,0.5) {$p_1$};
\node[above] at (3,2) {$n_1$};
\node[below] at (2,-1) {$m_2$};
\node at (2,0.5) {$p_2$};
\node[above] at (2,2) {$n_2$};
\node[text centered] at (0,0.5) {$\cdots$};
\node[text centered] at (1,0.5) {$\cdots$};
\draw[ultra thick,->] (-1,1) -- (0,1);
\end{tikzpicture},
\end{equation*}
where on the left, $x$ and $y$ indicate the row rapidities, and the left boundary conditions are given by the arrows as indicated.
\end{prop}

\begin{lem}[{\hspace{1sp}\cite[Lemma 4.11]{BBCW18}}]
\label{lem: boson HL}
Fix two partitions $\lambda$ and $\mu$ such that $\lambda=1^{m_1(\lambda)} 2^{m_2(\lambda)} \ldots$ and $\mu=1^{m_1(\mu)} 2^{m_2(\mu)} \ldots$. We then have
\begingroup
\allowdisplaybreaks
\begin{align*}
    w_x
\left(
\begin{tikzpicture}[baseline=(current bounding box.center),>=stealth,scale=0.8]
\draw[lgray,ultra thick] (0,0) -- (7,0);
\foreach\x in {1,...,6}{
\draw[lgray,line width=10pt] (7-\x,-1) -- (7-\x,1);
}
\node[left] at (0,0) {$0$};
\node[right] at (7,0) {$0$};
\foreach\x in {1,2,3}{
\node[text centered,below] at (7-\x,-1) {\tiny $m_{\x}(\lambda)$};
\node[text centered,above] at (7-\x,1) {\tiny $m_{\x}(\mu)$};
}
\foreach\x in {4,5}{
\node[text centered,below] at (7-\x,-1) {\tiny $\cdots$};
\node[text centered,above] at (7-\x,1) {\tiny $\cdots$};
}
\end{tikzpicture}
\right)&=\mathbf{1}_{l(\lambda)=l(\mu)}P_{\lambda/\mu}(x;0,t),
    \\w_x
\left(
\begin{tikzpicture}[baseline=(current bounding box.center),>=stealth,scale=0.8]
\draw[lgray,ultra thick] (0,0) -- (7,0);
\foreach\x in {1,...,6}{
\draw[lgray,line width=10pt] (7-\x,-1) -- (7-\x,1);
}
\node[left] at (0,0) {$0$};
\draw[ultra thick,->] (6,0) -- (7,0);
\node[right] at (7,0) {$1$};
\foreach\x in {1,2,3}{
\node[text centered,below] at (7-\x,-1) {\tiny $m_{\x}(\lambda)$};
\node[text centered,above] at (7-\x,1) {\tiny $m_{\x}(\mu)$};
}
\foreach\x in {4,5}{
\node[text centered,below] at (7-\x,-1) {\tiny $\cdots$};
\node[text centered,above] at (7-\x,1) {\tiny $\cdots$};
}
\end{tikzpicture}\right)&=\mathbf{1}_{l(\lambda)=l(\mu)+1}P_{\lambda/\mu}(x;0,t),
    \\w_x
\left(
\begin{tikzpicture}[baseline=(current bounding box.center),>=stealth,scale=0.8]
\draw[lred,ultra thick] (0,0) -- (7,0);
\foreach\x in {1,...,6}{
\draw[lred,line width=10pt] (7-\x,-1) -- (7-\x,1);
}
\node[left] at (0,0) {$1$};
\draw[ultra thick,->] (0,0) -- (1,0);
\node[right] at (7,0) {$0$};
\foreach\x in {1,2,3}{
\node[text centered,below] at (7-\x,-1) {\tiny $m_{\x}(\mu)$};
\node[text centered,above] at (7-\x,1) {\tiny $m_{\x}(\lambda)$};
}
\foreach\x in {4,5}{
\node[text centered,below] at (7-\x,-1) {\tiny $\cdots$};
\node[text centered,above] at (7-\x,1) {\tiny $\cdots$};
}
\end{tikzpicture}\right)&=\mathbf{1}_{l(\lambda)=l(\mu)+1}Q_{\lambda/\mu}(x;0,t),
    \\w_x
\left(
\begin{tikzpicture}[baseline=(current bounding box.center),>=stealth,scale=0.8]
\draw[lred,ultra thick] (0,0) -- (7,0);
\foreach\x in {1,...,6}{
\draw[lred,line width=10pt] (7-\x,-1) -- (7-\x,1);
}
\node[left] at (0,0) {$1$};
\draw[ultra thick,->] (0,0) -- (1,0);
\draw[ultra thick,->] (6,0) -- (7,0);
\node[right] at (7,0) {$1$};
\foreach\x in {1,2,3}{
\node[text centered,below] at (7-\x,-1) {\tiny $m_{\x}(\mu)$};
\node[text centered,above] at (7-\x,1) {\tiny $m_{\x}(\lambda)$};
}
\foreach\x in {4,5}{
\node[text centered,below] at (7-\x,-1) {\tiny $\cdots$};
\node[text centered,above] at (7-\x,1) {\tiny $\cdots$};
}
\end{tikzpicture}\right)&=\mathbf{1}_{l(\lambda)=l(\mu)}Q_{\lambda/\mu}(x;0,t),
\end{align*}
\endgroup
where the right hand sides of these expressions denote skew Hall--Littlewood polynomials.
\end{lem}

We also need the following lemma which allows powers of $q$ to be absorbed into the parameters $a$ and $b$.

\begin{lem}
\label{lem: u power shift}
Let $n_1,\dotsc$ and $m_1,\dotsc$ be finitely supported sequences of non-negative integers. We have the following equalities of one-row partition functions:
\begin{multline*}
    q^{\sum_i in_i} w_{qx}
\left(
\begin{tikzpicture}[baseline=(current bounding box.center),>=stealth,scale=0.8]
\draw[lgray,ultra thick] (0,0) -- (7,0);
\foreach\x in {1,...,6}{
\draw[lgray,line width=10pt] (7-\x,-1) -- (7-\x,1);
}
\node[left] at (0,0) {$0$};
\node[right] at (7,0) {$j$};
\foreach\x in {1,2,3}{
\node[text centered,below] at (7-\x,-1) {\tiny $m_{\x}$};
\node[text centered,above] at (7-\x,1) {\tiny $n_{\x}$};
}
\foreach\x in {4,5}{
\node[text centered,below] at (7-\x,-1) {\tiny $\cdots$};
\node[text centered,above] at (7-\x,1) {\tiny $\cdots$};
}
\end{tikzpicture}\right)
\\
=q^{\sum_i im_i}w_x
\left(
\begin{tikzpicture}[baseline=(current bounding box.center),>=stealth,scale=0.8]
\draw[lgray,ultra thick] (0,0) -- (7,0);
\foreach\x in {1,...,6}{
\draw[lgray,line width=10pt] (7-\x,-1) -- (7-\x,1);
}
\node[left] at (0,0) {$0$};
\node[right] at (7,0) {$j$};
\foreach\x in {1,2,3}{
\node[text centered,below] at (7-\x,-1) {\tiny $m_{\x}$};
\node[text centered,above] at (7-\x,1) {\tiny $n_{\x}$};
}
\foreach\x in {4,5}{
\node[text centered,below] at (7-\x,-1) {\tiny $\cdots$};
\node[text centered,above] at (7-\x,1) {\tiny $\cdots$};
}
\end{tikzpicture}\right),
\end{multline*}
and
\begin{multline*}
    u^{\sum_i in_i} w_{x}
\left(
\begin{tikzpicture}[baseline=(current bounding box.center),>=stealth,scale=0.8]
\draw[lred,ultra thick] (0,0) -- (7,0);
\foreach\x in {1,...,6}{
\draw[lred,line width=10pt] (7-\x,-1) -- (7-\x,1);
}
\node[left] at (0,0) {$1$};
\node[right] at (7,0) {$j$};
\foreach\x in {1,2,3}{
\node[text centered,below] at (7-\x,-1) {\tiny $m_{\x}$};
\node[text centered,above] at (7-\x,1) {\tiny $n_{\x}$};
}
\foreach\x in {4,5}{
\node[text centered,below] at (7-\x,-1) {\tiny $\cdots$};
\node[text centered,above] at (7-\x,1) {\tiny $\cdots$};
\draw[ultra thick,->] (0,0) -- (1,0);
}
\end{tikzpicture}\right)
\\
=u^{\sum_i im_i}w_{ux}
\left(
\begin{tikzpicture}[baseline=(current bounding box.center),>=stealth,scale=0.8]
\draw[lred,ultra thick] (0,0) -- (7,0);
\foreach\x in {1,...,6}{
\draw[lred,line width=10pt] (7-\x,-1) -- (7-\x,1);
}
\node[left] at (0,0) {$1$};
\node[right] at (7,0) {$j$};
\foreach\x in {1,2,3}{
\node[text centered,below] at (7-\x,-1) {\tiny $m_{\x}$};
\node[text centered,above] at (7-\x,1) {\tiny $n_{\x}$};
}
\foreach\x in {4,5}{
\node[text centered,below] at (7-\x,-1) {\tiny $\cdots$};
\node[text centered,above] at (7-\x,1) {\tiny $\cdots$};
\draw[ultra thick,->] (0,0) -- (1,0);
}
\end{tikzpicture}\right).
\end{multline*}
\end{lem}

\subsubsection{Boundary compatibility}

We will introduce an extra feature which we call a \emph{boundary vertex}, which corresponds to the boundary of the half space six vertex model. We will use a dot on a horizontal edge to denote it visually, and it has the effect of probabilistically outputting an arrow depending on the input, with probabilities
\begin{align*}
\begin{array}{cccc}
\begin{tikzpicture}[scale=0.8,>=stealth]
\draw[lgray,ultra thick] (-1,0) -- (1,0);
\node at (0,0) {$\bullet$};
\end{tikzpicture}
\quad\quad
&
\begin{tikzpicture}[scale=0.8,>=stealth]
\draw[lgray,ultra thick] (-1,0) -- (1,0);
\draw[ultra thick,->] (0,0) -- (1,0);
\node at (0,0) {$\bullet$};
\end{tikzpicture}
\quad\quad
&
\begin{tikzpicture}[scale=0.8,>=stealth]
\draw[lgray,ultra thick] (-1,0) -- (1,0);
\draw[ultra thick,->] (-1,0) -- (0,0);
\node at (0,0) {$\bullet$};
\end{tikzpicture}
\quad\quad
&
\begin{tikzpicture}[scale=0.8,>=stealth]
\draw[lgray,ultra thick] (-1,0) -- (1,0);
\draw[ultra thick,->] (-1,0) -- (1,0);
\node at (0,0) {$\bullet$};
\end{tikzpicture}
\vspace{0.2cm}
\\
1+abh_1(x)
\quad\quad
&
-abh_1(x)
\quad\quad
&
h_1(x)
\quad\quad
&
1-h_1(x)
\end{array}\end{align*}
where $x$ is the row rapidity. Note these are equal to the boundary vertex weights in the half space six vertex model given by \eqref{Stochastic boundary K-weight table}. We also define
\begin{align*}
\begin{array}{cccc}
\begin{tikzpicture}[scale=0.8,>=stealth]
\draw[lgray,ultra thick] (-1,0) -- (1,0);
\node[red] at (0,0) {$\bullet$};
\end{tikzpicture}
\quad\quad
&
\begin{tikzpicture}[scale=0.8,>=stealth]
\draw[lgray,ultra thick] (-1,0) -- (1,0);
\draw[ultra thick,->] (0,0) -- (1,0);
\node[red] at (0,0) {$\bullet$};
\end{tikzpicture}
\quad\quad
&
\begin{tikzpicture}[scale=0.8,>=stealth]
\draw[lgray,ultra thick] (-1,0) -- (1,0);
\draw[ultra thick,->] (-1,0) -- (0,0);
\node[red] at (0,0) {$\bullet$};
\end{tikzpicture}
\quad\quad
&
\begin{tikzpicture}[scale=0.8,>=stealth]
\draw[lgray,ultra thick] (-1,0) -- (1,0);
\draw[ultra thick,->] (-1,0) -- (1,0);
\node[red] at (0,0) {$\bullet$};
\end{tikzpicture}
\vspace{0.2cm}
\\
1-h_2(x)
\quad\quad
&
h_2(x)
\quad\quad
&
-cdh_2(x)
\quad\quad
&
1+cdh_2(x)
\end{array}\end{align*}
which correspond to the boundary weights \eqref{Stochastic boundary dual K-weight table}.

Note that due to the differences in the weights for the even and odd columns, we must pass the boundary vertices through two at a time. The first part of this lemma appeared already in \cite{H23}, and the second part follows from the first by conjugation of the weights (or a direct verification like in \cite{H23}).

\begin{lem}
	\label{lem: two vertex compat}
	For any choice of $i,j=0,1$ and $n_1,n_2\in\N$, we have
	\begin{multline*}
	\sum_{m_1,m_2=0}^{\infty}h_{m_2}(ab;t) a^{m_1}h_{m_1}\left(\frac{b}{a};t\right)
	w_x\left(
	\begin{tikzpicture}[scale=0.7,>=stealth,baseline=(current bounding box.center)]
	\draw[lgray,ultra thick] (-1.5,0) -- (2,0);
	\draw[lgray,line width=10pt] (0,-1) -- (0,1);
	\draw[lgray,line width=10pt] (1,-1) -- (1,1);
	\node[below] at (0,-1) {$m_2$};
	\node[above] at (0,1) {$n_2$};
	\node[below] at (1,-1) {$m_1$};
	\node[above] at (1,1) {$n_1$};
	\node[left] at (-1.5,0) {$i$};
	\node[right] at (2,0) {$j$};
	\node at (-1,0) {$\bullet$};
	\end{tikzpicture}
	\right)
	\\=\sum_{m_1,m_2=0}^{\infty} h_{m_2}(ab;t) a^{m_1}h_{m_1}\left(\frac{b}{a};t\right) w_x\left(
	\begin{tikzpicture}[scale=0.7,>=stealth,baseline=(current bounding box.center)]
	\draw[lred,ultra thick] (-1,0) -- (2.5,0);
	\draw[lred,line width=10pt] (0,-1) -- (0,1);
	\draw[lred,line width=10pt] (1,-1) -- (1,1);
	\node[below] at (0,-1) {$m_2$};
	\node[above] at (0,1) {$n_2$};
	\node[below] at (1,-1) {$m_1$};
	\node[above] at (1,1) {$n_1$};
	\node[left] at (-1,0) {$i$};
	\node[right] at (2.5,0) {$j$};
	\node at (2,0) {$\bullet$};
	\end{tikzpicture}
	\right),
	\end{multline*}
	and for any choice of $i,j=0,1$ and $m_1,m_2\in\N$, we have
	\begin{multline*}
	\sum_{n_1,n_2=0}^{\infty}h_{n_2}(cd;t) c^{n_1}h_{n_1}\left(\frac{d}{c};t\right)\frac{1}{(t;t)_{n_1}(t;t)_{n_2}}
	w_x\left(
	\begin{tikzpicture}[scale=0.7,>=stealth,baseline=(current bounding box.center)]
	\draw[lred,ultra thick] (-1.5,0) -- (2,0);
	\draw[lred,line width=10pt] (0,-1) -- (0,1);
	\draw[lred,line width=10pt] (1,-1) -- (1,1);
	\node[below] at (0,-1) {$m_2$};
	\node[above] at (0,1) {$n_2$};
	\node[below] at (1,-1) {$m_1$};
	\node[above] at (1,1) {$n_1$};
	\node[left] at (-1.5,0) {$i$};
	\node[right] at (2,0) {$j$};
	\node[red] at (-1,0) {$\bullet$};
	\end{tikzpicture}
	\right)
	\\=\sum_{n_1,n_2=0}^{\infty} h_{n_2}(cd;t) c^{m_1}h_{n_1}\left(\frac{d}{c};t\right)\frac{1}{(t;t)_{n_1}(t;t)_{n_2}} w_x\left(
	\begin{tikzpicture}[scale=0.7,>=stealth,baseline=(current bounding box.center)]
	\draw[lgray,ultra thick] (-1,0) -- (2.5,0);
	\draw[lgray,line width=10pt] (0,-1) -- (0,1);
	\draw[lgray,line width=10pt] (1,-1) -- (1,1);
	\node[below] at (0,-1) {$m_2$};
	\node[above] at (0,1) {$n_2$};
	\node[below] at (1,-1) {$m_1$};
	\node[above] at (1,1) {$n_1$};
	\node[left] at (-1,0) {$i$};
	\node[right] at (2.5,0) {$j$};
	\node[red] at (2,0) {$\bullet$};
	\end{tikzpicture}
	\right),
	\end{multline*}
\end{lem}

By iterating Lemma \ref{lem: two vertex compat}, we obtain the following proposition.
\begin{prop}
	\label{prop: two vertex compat}
	Let $n_1,n_2,\dotsc$ be a finitely supported sequence of non-negative integers. We have
	\begin{multline*}
	\sum_{\lambda}h_\lambda(a,b;t)w_x\left(
	\begin{tikzpicture}[baseline=(current bounding box.center),>=stealth,scale=0.8]
	\node[left] (0,0) {$1$};
	\draw[lgray,ultra thick] (0,0) -- (8,0);
	\draw[ultra thick,->] (0,0) -- (1,0);
	\foreach\x in {0,...,5}{
		\draw[lgray,line width=10pt] (7-\x,-1) -- (7-\x,1);
	}
	\node at (1,0) {$\bullet$};
	\node[right] at (8,0) {$j$};
	\foreach\x in {1,2,3}{
		\node[text centered,below] at (8-\x,-1) {\tiny $m_{\x}(\lambda)$};
		\node[text centered,above] at (8-\x,1) {\tiny $n_{\x}$};
	}
	\foreach\x in {3,4}{
		\node[text centered,below] at (7-\x,-1) {\tiny $\cdots$};
		\node[text centered,above] at (7-\x,1) {\tiny $\cdots$};
	}
	\end{tikzpicture}
	\right)
	\\
	=
	\sum_{\lambda}h_\lambda(a,b;t)w_x
	\left(
	\begin{tikzpicture}[baseline=(current bounding box.center),>=stealth,scale=0.8]
	\draw[lred,ultra thick] (0,0) -- (8,0);
	\foreach\x in {1,...,6}{
		\draw[lred,line width=10pt] (7-\x,-1) -- (7-\x,1);
	}
	\node[left] at (0,0) {$1$};
	\node at (7,0) {$\bullet$};
	\draw[ultra thick,->] (0,0) -- (1,0);
	\draw[ultra thick,->] (6,0) -- (7,0);
	\node[right] at (8,0) {$j$};
	\foreach\x in {1,2,3}{
		\node[text centered,below] at (7-\x,-1) {\tiny $m_{\x}(\lambda)$};
		\node[text centered,above] at (7-\x,1) {\tiny $n_{\x}$};
	}
	\foreach\x in {4,5}{
		\node[text centered,below] at (7-\x,-1) {\tiny $\cdots$};
		\node[text centered,above] at (7-\x,1) {\tiny $\cdots$};
	}
	\end{tikzpicture}
	\right),
	\end{multline*}
	and
	\begin{multline*}
	\sum_{\mu}\frac{h_\mu(c,d;t)}{\prod (t;t)_{m_i(\mu)}}w_x\left(
	\begin{tikzpicture}[baseline=(current bounding box.center),>=stealth,scale=0.8]
	\node[left] (0,0) {$0$};
	\draw[lred,ultra thick] (0,0) -- (8,0);
	\foreach\x in {0,...,5}{
		\draw[lred,line width=10pt] (7-\x,-1) -- (7-\x,1);
	}
	\node[red] at (1,0) {$\bullet$};
	\node[right] at (8,0) {$j$};
	\foreach\x in {1,2,3}{
		\node[text centered,below] at (8-\x,-1) {\tiny $n_{\x}$};
		\node[text centered,above] at (8-\x,1) {\tiny $m_{\x}(\mu)$};
	}
	\foreach\x in {3,4}{
		\node[text centered,below] at (7-\x,-1) {\tiny $\cdots$};
		\node[text centered,above] at (7-\x,1) {\tiny $\cdots$};
		\draw[ultra thick,->] (1,0) -- (2,0);
	}
	\end{tikzpicture}
	\right)
	\\
	=
	\sum_{\mu}\frac{h_\mu(c,d;t)}{\prod (t;t)_{m_i(\mu)}}w_x
	\left(
	\begin{tikzpicture}[baseline=(current bounding box.center),>=stealth,scale=0.8]
	\draw[lgray,ultra thick] (0,0) -- (8,0);
	\foreach\x in {1,...,6}{
		\draw[lgray,line width=10pt] (7-\x,-1) -- (7-\x,1);
	}
	\node[left] at (0,0) {$0$};
	\node[red] at (7,0) {$\bullet$};
	\draw[ultra thick,->] (6,0) -- (7,0);
	\node[right] at (8,0) {$j$};
	\foreach\x in {1,2,3}{
		\node[text centered,below] at (7-\x,-1) {\tiny $n_{\x}$};
		\node[text centered,above] at (7-\x,1) {\tiny $m_{\x}(\mu)$};
	}
	\foreach\x in {4,5}{
		\node[text centered,below] at (7-\x,-1) {\tiny $\cdots$};
		\node[text centered,above] at (7-\x,1) {\tiny $\cdots$};
	}
	\end{tikzpicture}
	\right),
	\end{multline*}
	where the sums are over partitions $\lambda$.
\end{prop}

\subsection{Proof of Theorem \ref{thm: HL 6vm}}

Recall that
\begin{equation*}
    G(z)=\frac{(abcdz^2;u,t)_\infty}{(uz;u)_\infty (tz;t)_\infty (utz;u,t)_\infty(abcdz;u,t)_\infty }\prod_{s\in \left\{\substack{ab,ac,ad,\\bc,bd,cd}\right\}}\frac{1}{(sz;u,t)_\infty}.
\end{equation*}

\begin{lem}
	\label{lem: rand shift pgf}
	We have
	\begin{equation*}
	\sum_{\lambda}\frac{h_\lambda(a,b;t)h_\lambda(c/\sqrt{u},d/\sqrt{u};t)z^{l(\lambda)}u^{|\lambda|/2}}{\prod_i(t;t)_{m_i(\lambda)}}=G(z).
	\end{equation*}
\end{lem}
\begin{proof}
	We recall Mehler's formula for Rogers--Szeg\"o polynomials, which states that (see e.g. \cite{C72})
	\begin{equation*}
	\sum_{k}h_k(a;t)h_k(b;t)\frac{z^k}{(t;t)_k}=\frac{(abz^2;t)_\infty}{(z;t)_\infty(az;t)_\infty(bz;t)_\infty(abz;t)_\infty}.
	\end{equation*}
	We then rewrite the sum over $\lambda$ as a sum over $m_i=m_i(\lambda)$, obtaining a product over $i$ of sums of this form. The even factors are of the form
	\begin{equation*}
	\sum_{m_i}h_{m_i}(ab;t)h_{m_i}(cd/u;t)\frac{u^{im_i/2}z^{m_i}}{(t;t)_{m_i}}=\frac{(abcdu^{i-1}z^2;t)_\infty}{(u^{i/2}z;t)_\infty(abzu^{i/2};t)_\infty(cdzu^{(i-2)/2};t)_\infty(abcdzu^{(i-2)/2};t)_\infty}
	\end{equation*}
	and the odd factors are of the form
	\begin{equation*}
 \begin{split}
	&\sum_{m_i}h_{m_i}(a/b;t)h_{m_i}(c/d;t)\frac{(bd)^{m_i}u^{(i-1)m_i/2}z^{m_i}}{(t;t)_{m_i}}
 \\=&\frac{(abcdu^{i-1}z^2;t)_\infty}{(bdu^{(i-1)/2}z;t)_\infty(adzu^{(i-1)/2};t)_\infty(cbzu^{(i-1)/2};t)_\infty(aczu^{(i-1)/2};t)_\infty}.
\end{split}
	\end{equation*}
	Multiplying over all $i$ gives the desired equality.
\end{proof}

\begin{proof}[Proof of Theorem \ref{thm: HL 6vm}]
	We have by Lemma \ref{lem: boson HL} that $\mathbb{FBHL}_a^{(q,t,\nu)}(\ell(\lambda^{(0)})=n,[\vec{\lambda}]=s)$ is equal to
	\begin{equation*}
	\frac{1}{\Phi_{HL}}\sum_{\ell(\mu)=n}\frac{h_{\lambda}(a,b;t)h_\mu(c/\sqrt{q},d/\sqrt{q};t)q^{|\mu|/2}}{\prod_i (t;t)_{m_i(\mu)}}
	\begin{tikzpicture}[scale=0.8,baseline=(current bounding box.center),>=stealth]
	\foreach\x in {0,...,6}{
		\draw[lgray,line width=10pt] (\x,0) -- (\x,7);
	}
	\foreach\y in {1,...,6}{
		\draw[lgray,thick] (-1,\y) -- (7,\y);
	}
	\draw[ultra thick,->] (6,1) -- (7,1); \draw[ultra thick,->] (6,2) -- (7,2);
	\draw[ultra thick,->] (6,4) -- (7,4); \draw[ultra thick,->] (6,6) -- (7,6);
	\node[right] at (7,1) {$s_N$}; \node[right] at (7,6) {$s_1$};
	\node[left] at (-1,6) {$x_1$};
	\node[left] at (-1,3.5) {$\vdots$};
	\node[left] at (8,3.5) {$\vdots$};
	\node[left] at (-1,1) {$x_N$};
	\node[below] at (6,0) {\tiny $m_1(\lambda)$};
	\node[below] at (5,0) {\tiny $m_2(\lambda)$};
	\node[below] at (4,0) {\tiny $m_3(\lambda)$};
	\node[below] at (2,0) {$\cdots$};
	\node[above] at (6,7) {\tiny $m_1(\mu)$};
	\node[above] at (5,7) {\tiny $m_2(\mu)$};
	\node[above] at (4,7) {\tiny $m_3(\mu)$};
	\node[above] at (2,7) {$\cdots$};
	\end{tikzpicture}
	\end{equation*}
	where the boundary conditions on the left are empty (the $a_i$ indicate the row rapidities) and the boundary conditions on the right are indicated by the $s_i$. Here, we write $\lambda^{(0)}=\mu$ and $\lambda^{(N)}=\lambda$.
	
	We can introduce a boundary vertex at the left on the bottom row, at the cost of a $\frac{1-x_n^2}{(1-ax_n)(1+bx_n)}$ factor (which will cancel with a factor in $\Phi_{HL}$), since if an arrow enters the row the weight will be $0$ so only one outcome is possible. We then use Proposition \ref{prop: two vertex compat} to move the boundary vertex to the right, resulting in the expression
	\begin{equation*}
	\begin{array}{l}
	\displaystyle
	\frac{1-x_n^2}{(1+ax_n)(1+bx_n)}\frac{1}{\Phi_{HL}}
	\\
	\\\displaystyle
	\qquad\times \sum_{\ell(\mu)=n}\frac{h_\lambda(a,b;t)h_\mu(c/\sqrt{q},d/\sqrt{q};t)q^{|\mu|/2}}{\prod_i (t;t)_{m_i(\mu)}}
	\end{array}
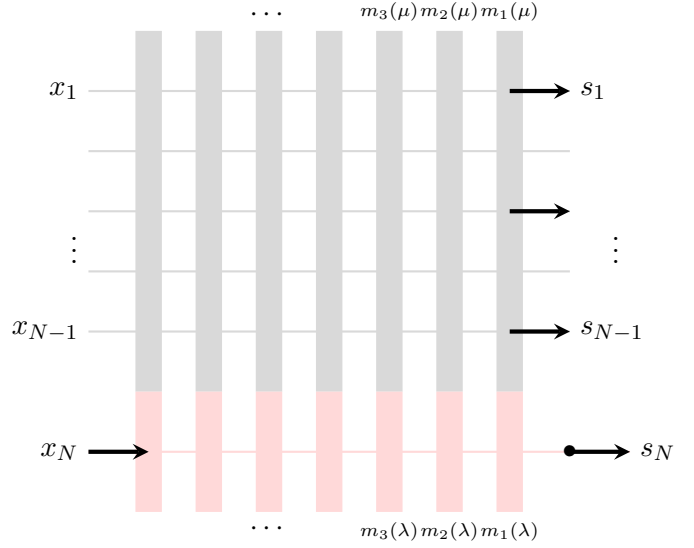
	\begin{tikzpicture}[scale=0.8,baseline=(current bounding box.center),>=stealth]
	\foreach\x in {0,...,6}{
		\draw[lgray,line width=10pt] (\x,1) -- (\x,7);
	}
	\foreach\y in {2,...,6}{
		\draw[lgray,thick] (-1,\y) -- (7,\y);
	}
	\foreach\x in {0,...,6}{
		\draw[lred,line width=10pt] (\x,-1) -- (\x,1);
	}
	\foreach\y in {0}{
		\draw[lred,thick] (-1,\y) -- (8,\y);
	}
	\node at (7,0) {$\bullet$};
	\draw[ultra thick,->] (-1,0) -- (0,0);
	\draw[ultra thick,->] (7,0) -- (8,0); \draw[ultra thick,->] (6,2) -- (7,2);
	\draw[ultra thick,->] (6,4) -- (7,4); \draw[ultra thick,->] (6,6) -- (7,6);
	\node[right] at (7,2) {$s_{N-1}$};
	\node[right] at (8,0) {$s_N$}; \node[right] at (7,6) {$s_1$};
	\node[left] at (-1,6) {$x_1$};
	\node[left] at (-1,3.5) {$\vdots$};
	\node[left] at (8,3.5) {$\vdots$};
	\node[left] at (-1,2) {$x_{N-1}$};
	\node[left] at (-1,0) {$x_N$};
	\node[below] at (6,-1) {\tiny $m_1(\lambda)$};
	\node[below] at (5,-1) {\tiny $m_2(\lambda)$};
	\node[below] at (4,-1) {\tiny $m_3(\lambda)$};
	\node[below] at (2,-1) {$\cdots$};
	\node[above] at (6,7) {\tiny $m_1(\mu)$};
	\node[above] at (5,7) {\tiny $m_2(\mu)$};
	\node[above] at (4,7) {\tiny $m_3(\mu)$};
	\node[above] at (2,7) {$\cdots$};
	\end{tikzpicture}
	\end{equation*}
	
	Next, we use Proposition \ref{prop: YB boson} to swap the bottom row all the way to the top, at the cost of a $\prod _i\frac{1-x_ix_n}{1-tx_ix_n}$ factor, coming from $\Phi_{HL}$, resulting in the expression
	\begin{multline*}
	\frac{1-x_n^2}{(1+ax_n)(1+bx_n)}\prod_{i<n}\frac{1-x_ix_n}{1-tx_ix_n}\frac{1}{\Phi_{HL}} \sum_{\ell(\mu)=n}\frac{h_\lambda(a,b;t)h_\mu(c/\sqrt{q},d/\sqrt{q};t)q^{|\mu|/2}}{\prod_i (t;t)_{m_i(\mu)}}
    \\
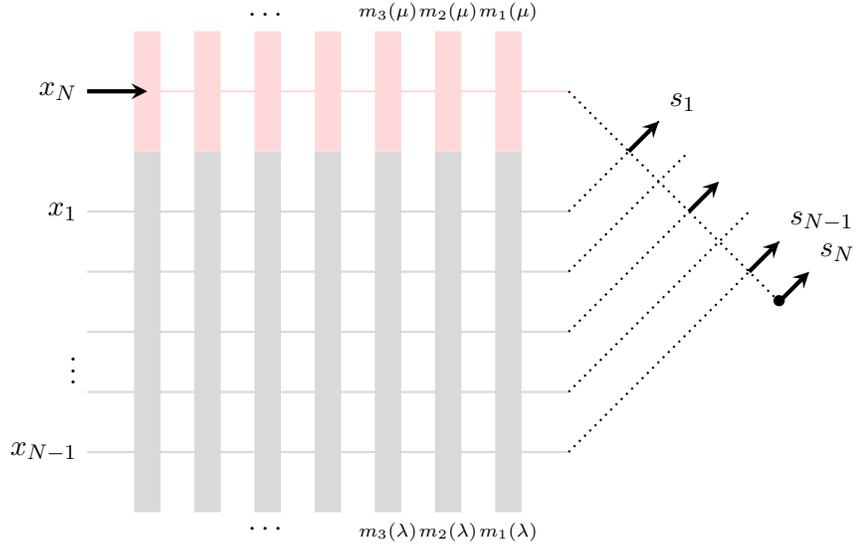
	\begin{tikzpicture}[scale=0.8,baseline=(current bounding box.center),>=stealth]
	\foreach\x in {0,...,6}{
		\draw[lgray,line width=10pt] (\x,-1) -- (\x,5);
	}
	\foreach\y in {0,...,4}{
		\draw[lgray,thick] (-1,\y) -- (7,\y);
	}
	\foreach\x in {0,...,6}{
		\draw[lred,line width=10pt] (\x,5) -- (\x,7);
	}
	\foreach\y in {6}{
		\draw[lred,thick] (-1,\y) -- (7,\y);
	}
	\draw[thick, dotted] (7,6) -- (10.5,2.5);
	\node at (10.5,2.5) {$\bullet$}; \draw[ultra thick,->] (10.5,2.5) -- (11,3);
	\draw[thick, dotted] (7,4) -- (8.5,5.5) node[above right] {$s_1$};
	\draw[thick, dotted] (7,3) -- (9,5);
	\draw[thick, dotted] (7,2) -- (9.5,4.5);
	\draw[thick, dotted] (7,1) -- (10,4);
	\draw[thick, dotted] (7,0) -- (10.5,3.5) node[above right] {$s_{N-1}$};
	\draw[thick, dotted] (10.5,2.5) -- (11,3) node[above right] {$s_N$};
	\draw[ultra thick,->] (8,5) -- (8.5,5.5); \draw[ultra thick,->] (9,4) -- (9.5,4.5);
	\draw[ultra thick,->] (10,3) -- (10.5,3.5);
	\draw[ultra thick,->] (-1,6) -- (0,6);
	\node[left] at (-1,4) {$x_1$};
	\node[left] at (-1,1.5) {$\vdots$};
	\node[left] at (-1,0) {$x_{N-1}$};
	\node[left] at (-1,6) {$x_N$};
	\node[below] at (6,-1) {\tiny $m_1(\lambda)$};
	\node[below] at (5,-1) {\tiny $m_2(\lambda)$};
	\node[below] at (4,-1) {\tiny $m_3(\lambda)$};
	\node[below] at (2,-1) {$\cdots$};
	\node[above] at (6,7) {\tiny $m_1(\mu)$};
	\node[above] at (5,7) {\tiny $m_2(\mu)$};
	\node[above] at (4,7) {\tiny $m_3(\mu)$};
	\node[above] at (2,7) {$\cdots$};
	\end{tikzpicture}
	\end{multline*}

	Repeating this process, we obtain
	\begin{multline*}
	\prod_{i=1}^n\frac{1-x_i^2}{(1+ax_i)(1+ba_i)}\prod_{i<j}\frac{1-x_ix_j}{1-tx_ix_j}\frac{1}{\Phi_{HL}}\sum_{\ell(\lambda)+H_1=n}\frac{h_\lambda(a,b;t)h_\mu(c/\sqrt{q},d/\sqrt{q};t)q^{|\lambda|/2}}{\prod_i (t;t)_{m_i(\mu)}}
	\\
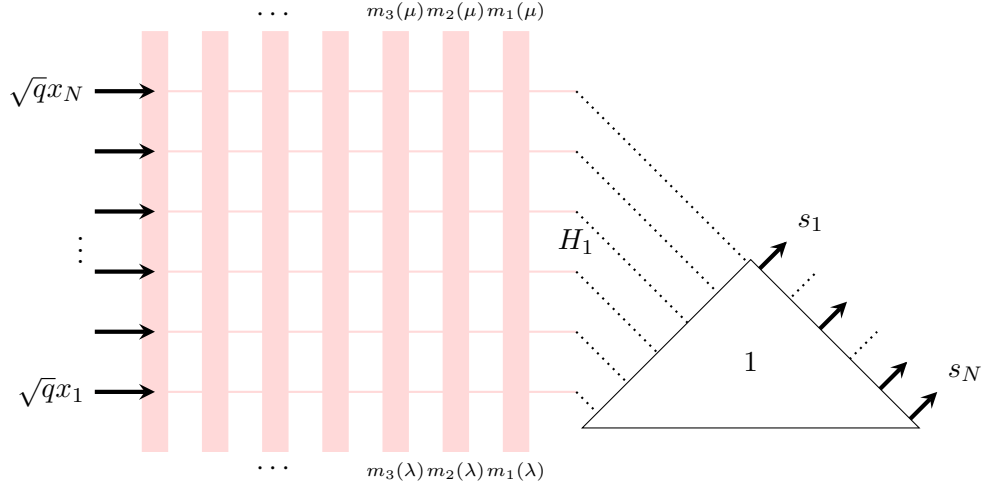
	\begin{tikzpicture}[scale=0.8,baseline=(current bounding box.center),>=stealth]
	\foreach\x in {0,...,6}{
		\draw[lred,line width=10pt] (\x,0) -- (\x,7);
	}
	\foreach\y in {1,...,6}{
		\draw[lred,thick] (-1,\y) -- (7,\y);
		\draw[ultra thick,->] (-1,\y) -- (0,\y);
	}
	\draw[thick, dotted] (7,6) -- (12.5,0.5); 
	\draw[thick, dotted] (12.5,0.5) -- (13,1) node[above right] {$s_N$}; \draw[ultra thick,->] (12.5,0.5) -- (13,1);
	\draw[thick, dotted] (7,5) -- (11.5,0.5);
	\draw[thick, dotted] (11.5,0.5) -- (12.5,1.5); \draw[ultra thick,->] (12,1) -- (12.5,1.5);
	\draw[thick, dotted] (7,4) -- (10.5,0.5); 
	\draw[thick, dotted] (10.5,0.5) -- (12,2);
	\draw[thick, dotted] (7,3) -- (9.5,0.5); 
	\draw[thick, dotted] (9.5,0.5) -- (11.5,2.5); \draw[ultra thick,->] (11,2) -- (11.5,2.5);
	\draw[thick, dotted] (7,2) -- (8.5,0.5);
	\draw[thick, dotted] (8.5,0.5) -- (11,3);
	\draw[thick, dotted] (7,1) -- (7.5,0.5);
	\draw[thick, dotted] (7.5,0.5) -- (10.5,3.5) node[above right] {$s_1$}; \draw[ultra thick,->] (10,3) -- (10.5,3.5);
	\node[left] at (-1,6) {$\sqrt{q}x_N$};
	\node[left] at (-1,3.5) {$\vdots$};
	\node[left] at (-1,1) {$\sqrt{q}x_1$};
	\node[below] at (6,0) {\tiny $m_1(\lambda)$};
	\node[below] at (5,0) {\tiny $m_2(\lambda)$};
	\node[below] at (4,0) {\tiny $m_3(\lambda)$};
	\node[below] at (2,0) {$\cdots$};
	\node[above] at (6,7) {\tiny $m_1(\mu)$};
	\node[above] at (5,7) {\tiny $m_2(\mu)$};
	\node[above] at (4,7) {\tiny $m_3(\mu)$};
	\node[above] at (2,7) {$\cdots$};
	\node at (7,3.5){$H_1$};
	\draw[fill=white](12.7,0.4)--(7.1,0.4)--(9.9,3.2) --cycle;\draw (9.9,1.5)node[]{$1$};
	\end{tikzpicture}
	\end{multline*}
where we have also used Lemma \ref{lem: u power shift} to shift the power in $q$, and $H_1+\ell(\lambda)=\ell(\mu)$, where $H_1$ is the number of edges leaving the right of the boson portion of the model without arrows.
	
Notice that now the sum over $\mu$ is free, and so we may again introduce a boundary vertex, this time at the top, and use Proposition \ref{prop: two vertex compat} to move it to the right, before using Proposition \ref{prop: YB boson} to move it to the bottom. Doing so, we obtain

	\begin{multline*}
    \frac{1-qx_n^2}{(1+cx_n)(1+dx_n)}\prod_{i<n}\frac{1-qx_ix_n}{1-qtx_ix_n}
    \\
\times\prod_{i=1}^n\frac{1-a_i^2}{(1-a_i\nu t)(1+a_i/\nu)}\prod_{i<j}\frac{1-a_ia_j}{1-qa_ia_j}\frac{1}{\Phi_{HL}}\sum_{\ell(\lambda)+H_1=n}\frac{h_\lambda(a,b;t)h_\mu(c/\sqrt{q},d/\sqrt{q};t)q^{|\lambda|/2}}{\prod_i (t;t)_{m_i(\mu)}}
\\
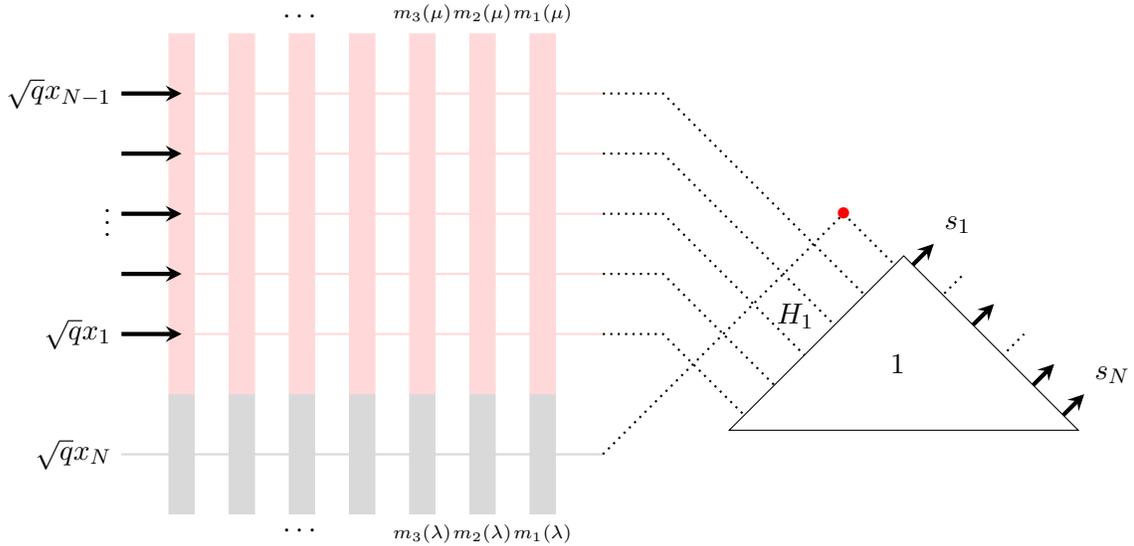
\begin{tikzpicture}[scale=0.8,baseline=(current bounding box.center),>=stealth]
\foreach\x in {0,...,6}{
	\draw[lred,line width=10pt] (\x,1) -- (\x,7);
}
\foreach\x in {0,...,6}{
	\draw[lgray,line width=10pt] (\x,-1) -- (\x,1);
}
\foreach\y in {2,...,6}{
	\draw[lred,thick] (-1,\y) -- (7,\y);
	\draw[ultra thick,->] (-1,\y) -- (0,\y);
}
\draw[lgray,thick](-1,0)--(7,0);
\draw[thick, dotted] (11,4) -- (14.5,0.5); 
\draw[thick, dotted] (14.5,0.5) -- (15,1) node[above right] {$s_N$}; \draw[ultra thick,->] (14.5,0.5) -- (15,1);
\draw[thick, dotted] (7,6)--(8,6) -- (13.5,0.5);
\draw[thick, dotted] (13.5,0.5) -- (14.5,1.5); \draw[ultra thick,->] (14,1) -- (14.5,1.5);
\draw[thick, dotted] (7,5)--(8,5) -- (12.5,0.5); 
\draw[thick, dotted] (12.5,0.5) -- (14,2);
\draw[thick, dotted] (7,4)--(8,4) -- (11.5,0.5); 
\draw[thick, dotted] (11.5,0.5) -- (13.5,2.5); \draw[ultra thick,->] (13,2) -- (13.5,2.5);
\draw[thick, dotted] (7,3)--(8,3) -- (10.5,0.5);
\draw[thick, dotted] (10.5,0.5) -- (13,3);
\draw[thick, dotted] (7,2)--(8,2) -- (9.5,0.5);
\draw[thick, dotted] (9.5,0.5) -- (12.5,3.5) node[above right] {$s_1$}; \draw[ultra thick,->] (12,3) -- (12.5,3.5);
\draw[thick, dotted](7,0)--(11,4)node[red]{$\bullet$};
\node[left] at (-1,0) {$\sqrt{q}x_{N}$};
\node[left] at (-1,6) {$\sqrt{q}x_{N-1}$};
\node[left] at (-1,4) {$\vdots$};
\node[left] at (-1,2) {$\sqrt{q}x_1$};
\node[below] at (6,-1) {\tiny $m_1(\lambda)$};
\node[below] at (5,-1) {\tiny $m_2(\lambda)$};
\node[below] at (4,-1) {\tiny $m_3(\lambda)$};
\node[below] at (2,-1) {$\cdots$};
\node[above] at (6,7) {\tiny $m_1(\mu)$};
\node[above] at (5,7) {\tiny $m_2(\mu)$};
\node[above] at (4,7) {\tiny $m_3(\mu)$};
\node[above] at (2,7) {$\cdots$};
\node at (10.2,2.3){$H_1$};
\draw[fill=white](14.9,0.4)--(9.1,0.4)--(12,3.3) --cycle;\draw (11.9,1.5)node[]{$1$};
\end{tikzpicture}
\end{multline*}

After doing this once for each row (essentially repeating the process from before except top to bottom), we obtain

\begin{multline*}
\prod_{i=1}^n\frac{1-x_i^2}{(1+ax_i\nu t)(1+bx_i)}\prod_{i<j}\frac{1-x_ix_j}{1-tx_ix_j}\prod_{i=1}^n\frac{1-qx_i^2}{(1+cx_i\nu t)(1+dx_i)}\prod_{i<j}\frac{1-qx_ix_j}{1-qtx_ix_j}
\\\times\frac{1}{\Phi_{HL}}\sum_{\ell(\mu)+H_1+V_1=n}\frac{h_\lambda(a,b;t)h_\mu(c/\sqrt{q},d/\sqrt{q};t)q^{|\lambda|/2}}{\prod_i (t;t)_{m_i(\mu)}}
\\
\begin{tikzpicture}[scale=0.8,baseline=(current bounding box.center),>=stealth]
\foreach\x in {0,...,6}{
	\draw[lgray,line width=10pt] (\x,0) -- (\x,7);
}
\foreach\y in {1,...,6}{
	\draw[lgray,thick] (-1,\y) -- (7,\y);
}
\draw[thick, dotted] (12.5,6.5) -- (16,3); 
\draw[thick, dotted] (16,3) -- (16.5,3.5) node[above right] {$s_N$}; \draw[ultra thick,->] (16,3) -- (16.5,3.5);
\draw[thick, dotted] (7,2)--(11.5,6.5) -- (15,3);
\draw[thick, dotted] (15,3) -- (16,4); \draw[ultra thick,->] (15.5,3.5) -- (16,4);
\draw[thick, dotted] (7,3)--(10.5,6.5) -- (14,3); 
\draw[thick, dotted] (14,3) -- (15.5,4.5);
\draw[thick, dotted] (7,4)--(9.5,6.5) -- (13,3); 
\draw[thick, dotted] (13,3) -- (15,5); \draw[ultra thick,->] (14.5,4.5) -- (15,5);
\draw[thick, dotted] (7,5)--(8.5,6.5)--(9.5,5.5) -- (12,3);
\draw[thick, dotted] (12,3) -- (14.5,5.5);
\draw[thick, dotted] (7,6)--(7.5,6.5) -- (11,3);
\draw[thick, dotted] (11,3) -- (14,6) node[above right] {$s_1$}; \draw[ultra thick,->] (13.5,5.5) -- (14,6);
\draw[thick, dotted](7,1)--(12.5,6.5);
\node[left] at (-1,6) {$qx_1$};
\node[left] at (-1,3.5) {$\vdots$};
\node[left] at (-1,1) {$qx_N$};
\node[below] at (6,0) {\tiny $m_1(\lambda)$};
\node[below] at (5,0) {\tiny $m_2(\lambda)$};
\node[below] at (4,0) {\tiny $m_3(\lambda)$};
\node[below] at (2,0) {$\cdots$};
\node[above] at (6,7) {\tiny $m_1(\mu)$};
\node[above] at (5,7) {\tiny $m_2(\mu)$};
\node[above] at (4,7) {\tiny $m_3(\mu)$};
\node[above] at (2,7) {$\cdots$};
\node at (11.8,4.7){$H_1$};
\node at (8.2,4.7){$V_1$};
\draw[fill=white](16.4,2.9)--(10.6,2.9)--(13.5,5.8) --cycle;\draw (13.5,4)node[]{$1$};
\draw[fill=white](12.9,6.6)--(7.1,6.6)--(10,3.7) --cycle;\draw (10,5.5)node[]{$\sqrt{q}$};
\end{tikzpicture}
\end{multline*}
where again we use Lemma \ref{lem: u power shift} to shift the power of $q$ and the fact that $\ell(\lambda)=\ell(\mu)+V_1$, where $V_1$ denotes the number of arrows exiting the boson portion of the model.

At this point, we recognize that the boson portion of the model is exactly the same as before, and so we can simply repeat the process $L$ times, obtaining

\begin{multline*}
\prod_{i=1}^n \frac{(x_i^2;q)_{2L}}{(-ax_i;q)_L(-bx_i;q)_L(-cx_i;q)_L(-dx_i;q)_L}\prod_{i<j}\frac{(x_ix_j;q)_{2L}}{(tx_ix_j;q)_{2L}}
\\\times\sum_{\ell(\mu)+H_1+V_1+\dotsc+H_L+V_L=n}\frac{h_\lambda(a,b;t)h_\mu(c/\sqrt{q},d/\sqrt{q};t)q^{|\lambda|/2}}{\prod_i (t;t)_{m_i(\mu)}}
\\
\begin{tikzpicture}[scale=0.8,baseline=(current bounding box.center),>=stealth]
\foreach\x in {2,...,6}{
	\draw[lgray,line width=10pt] (\x,0) -- (\x,7);
}
\foreach\y in {1,...,6}{
	\draw[lgray,thick] (1,\y) -- (7,\y);
}
\draw[thick, dotted](7,1)--(12.5,6.5);
\draw[thick, dotted] (12.5,6.5) --(13,6);
\draw[thick, dotted](16,6)-- (19,3); 
\draw[thick, dotted] (19,3) -- (19.5,3.5) node[above right] {$s_N$}; \draw[ultra thick,->] (19,3) -- (19.5,3.5);
\draw[thick, dotted] (7,2)--(11.5,6.5)--(12.5,5.5);
\draw[thick, dotted](15.5,5.5) -- (18,3);
\draw[thick, dotted] (18,3) -- (19,4); \draw[ultra thick,->] (18.5,3.5) -- (19,4);
\draw[thick, dotted] (7,3)--(10.5,6.5) -- (12,5);
\draw[thick,dotted](15,5)--(17,3); 
\draw[thick, dotted] (17,3) -- (18.5,4.5);
\draw[thick, dotted] (7,4)--(9.5,6.5) -- (11.5,4.5);
\draw[thick, dotted](14.5, 4.5)--(16,3); 
\draw[thick, dotted] (16,3) -- (18,5); \draw[ultra thick,->] (17.5,4.5) -- (18,5);
\draw[thick, dotted] (7,5)--(8.5,6.5)--(11,4);
\draw[thick, dotted](14,4) -- (15,3);
\draw[thick, dotted] (15,3) -- (17.5,5.5);
\draw[thick, dotted] (7,6)--(7.5,6.5)--(10.5,3.5);
\draw[thick,dotted](13.5,3.5) -- (14,3);
\draw[thick, dotted] (14,3) -- (17,6) node[above right] {$s_1$}; \draw[ultra thick,->] (16.5,5.5) -- (17,6);
\node[left] at (1,6) {$q^Lx_1$};
\node[left] at (1,3.5) {$\vdots$};
\node[left] at (1,1) {$q^Lx_N$};
\node[below] at (6,0) {\tiny $m_1(\lambda)$};
\node[below] at (5,0) {\tiny $m_2(\lambda)$};
\node[below] at (4,0) {\tiny $m_3(\lambda)$};
\node[below] at (3,0) {$\cdots$};
\node[above] at (6,7) {\tiny $m_1(\mu)$};
\node[above] at (5,7) {\tiny $m_2(\mu)$};
\node[above] at (4,7) {\tiny $m_3(\mu)$};
\node[above] at (3,7) {$\cdots$};
\node at (11.8,4.7){$H_L$};
\node at (8.2,4.7){$V_L$};
\draw[fill=white](19.4,2.9)--(13.6,2.9)--(16.5,5.8) --cycle;\draw (16.5,4)node[]{$1$};
\draw[fill=white](12.9,6.6)--(7.1,6.6)--(10,3.7) --cycle;\draw (10,5.5)node[]{$q^{L-\frac{1}{2}}$};
\node at (14.8,4.7){$H_1$};
\node at (13.3,4.7){$\dots$};
\end{tikzpicture}
\end{multline*}

Finally, we can send $L\to\infty$, and note that the row parameters will go to $0$ in the boson portion of the model. This has the effect of disallowing any turns, but arrows may still enter from the bottom and exit from the top. Moreover, the boson model and six vertex model portions decouple, in the sense that there is only one allowed configuration on the edges connecting the two, which is either that every edge is occupied, or no edge is occupied, depending on whether $L$ is odd or even.

The boson model portion then contributes according to the distribution given by Lemma \ref{lem: rand shift pgf}. We recognize the six vertex model portion as exactly the quasi-open six vertex model. We then see that the expression computes the probability that $S=s$ in the six vertex model portion, and that $l(\lambda)+H+V=n$, where $\lambda$ has probability generating function $G$ (and is independent of the six vertex model portion), and $H$ denotes the total number of horizontal edges between triangles which are unoccupied, and $V$ denotes the total number of vertical edges between triangles which are occupied, which is the desired result. In particular, the constants obtained in the above argument exactly cancel $\Phi_{HL}$.
\end{proof}

\subsection{Koornwinder evaluation for Hall--Littlewood measure}
We now state a bounded Littlewood identity which evaluates a bounded sum over skew Hall--Littlewood polynomials in terms of rectangular Koornwinder polynomials. This is essentially Theorem 3.14 of \cite{FV17}. In particular, the construction there is given in terms of $R$ and $K$ matrices which match ours. One could also prove this result using the techniques of \cite{RW21}, where a special case appears as Theorem 4.7.
\begin{thm}
\label{thm: koornwinder}
Fix an even integer $2n \geq 0$. We have
    \begin{equation*}
C_n\left(\prod_{i=1}^N x_i^n\right)K_{n^N}(x;q,t,a,b,c,d)
=
\sum_{\mu \subseteq \lambda:\lambda_1 \leq 2n}
\frac{h_{\lambda}(a,b;t)}{h_{m_{2n}(\lambda)}(ab;t)}
\cdot
P_{\lambda/\mu}(x;0,t)
\cdot
\frac{q^{|\mu|/2}
h_{\mu}(c/\sqrt{q},d/\sqrt{q};t)}{\prod_{i \geq 1}(t;t)_{m_i(\mu)}},
\end{equation*}
where
\begin{equation*}
C_n=C_n(q,t,a,b,c,d)=\frac{\prod_{i=n-1}^{2n-2}(abcdq^i;t)_\infty}{\prod_{i=1}^n(q^i;t)_\infty\prod_{i=1}^{n-1}(abq^i;t)_\infty\prod_{i=0}^{n-1}\prod_{s\in \{ac,ad,bc,bd,cd\}}(sq^i;t)_\infty}.
\end{equation*}
\end{thm}
\begin{proof}
This proof is essentially the one found in \cite{FV17}. We can express the right hand side (divided by $\prod_i x_i^n$) in terms of the boson model like in the proof of Theorem \ref{thm: HL 6vm}. 

The key tool is the $q$-KZ equations, which characterize the non-symmetric Koornwinder polynomials up to scaling, and in our setting requires that the action of the $R$ matrix \eqref{fund-vert} and $K$ matrices given by \eqref{Stochastic boundary K-weight table} and \eqref{Stochastic boundary dual K-weight table} correspond to the action of the affine Hecke algebra generated by $s_0,\dotsc, s_N$ where $s_i$ acts by swapping $x_i$ and $x_{i+1}$ for $i=1,\dotsc, N-1$, and $s_0$ and $s_N$ correspond to $x_1\mapsto qx_1^{-1}$ and $x_N\mapsto x_N^{-1}$. This can be verified using (a variant of) Proposition \ref{prop: YB boson} and Proposition \ref{prop: two vertex compat}. In particular, the $K$ matrices change the red and gray boson vertices, which essentially amounts to inverting the row parameters. For the $\mu$ boundary, we must first shift the $q$ weights to the bottom, so we are inverting $\sqrt{q}^{-1}x_1$ instead of $x_1$, which accounts for the extra factor of $q$. 

From there, we must check that the right hand side is invariant under the action of the BC Weyl group on the variables. It's clearly symmetric under permutations, and since the action of the $K$ matrices on the right has no effect (since we're summing over all configurations), it's symmetric under inversion.

Finally, we must compute the constant. The constant term on the left hand side equals $C_n$. We can compute the constant term on the right hand side by substituting $0$ for the $x_i$. We can then compute
\begin{equation*}
    \sum_{\lambda_1\leq 2n}\frac{q^{|\lambda|/2}h_\lambda(a,b;t)h_\lambda(c/\sqrt{q},d/\sqrt{q};t)}{h_{m_{2n}(\lambda)(ab;t)}\prod_{i}(t;t)_{m_i(\lambda)}}
\end{equation*}
as in the proof of Theorem \ref{thm:pf}.

We look at factors corresponding to $m_j(\lambda)$. The even factors evaluate to
\begin{equation*}
    \frac{(abcdq^{j-1};t)_\infty}{(q^{j/2};t)_\infty(q^{j/2}ab;t)_\infty(q^{j/2-1}cd;t)_\infty(q^{j/2-1}abcd;t)_\infty},
\end{equation*}
except the $j=2n$ factor where $ab=0$. The odd factors evaluate to
\begin{equation*}
    \frac{(abcdq^{j-1};t)_\infty}{(q^{\frac{j-1}{2}}bd;t)_\infty(q^{\frac{j-1}{2}}ad;t)_\infty(q^{\frac{j-1}{2}}ac;t)_\infty(q^{\frac{j-1}{2}}ac;t)_\infty}.
\end{equation*}
The product over $1\leq j\leq 2n$ gives the desired formula for $C_n$.

\end{proof}

\section*{Acknowledgements}
The authors would like to thank Amol Aggarwal, Guillaume Barraquand, and Alexei Borodin for helpful discussions. JH was partially supported by NSF Award DMS-2451487, as well as by the Swedish Research Council under grant no. 2021-06594 while the author was in residence at Institut Mittag-Leffler in Djursholm, Sweden during Fall 2024. MW was partially supported by the ARC Future Fellowship FT200100981. An extended abstract containing part of this work appeared in the proceedings of FPSAC 2025, see \cite{HW25}.

\bibliography{bibliography}{}
\bibliographystyle{alphaurl}

\end{document}